%% file: multi_linear_v6.tex
\documentclass[a4paper,12pt]{article}
\usepackage[utf8]{inputenc}

\usepackage{mathrsfs}
\usepackage{graphicx}
\usepackage{epstopdf}
\usepackage{enumerate}
\usepackage{multicol}
\usepackage{color}
\usepackage{amsmath,amssymb,amscd}
\usepackage{amsthm}

\usepackage{pdfpages}

\usepackage{times}
\usepackage[varg]{txfonts}

\usepackage[top=0.5in, bottom=0.5in, left=0.5in, right=0.5in]{geometry}

\theoremstyle{plain}
\newtheorem{thm}{Theorem}
\newtheorem{lem}{Lemma}
\newtheorem{cor}{Corollary}
\newtheorem{prop}{Proposition}

\theoremstyle{definition}
\newtheorem{dfn}{Definition}

\theoremstyle{remark}
\newtheorem*{rem}{Remark}
\newtheorem{opp}{Open Problem}
\newtheorem*{ackn}{Acknowledgment}

\title{On a class of linear functional equations\\ without range condition}
\author{Eszter Gselmann, Gergely Kiss and Csaba Vincze}

\begin{document}

\maketitle

\begin{abstract}
The main purpose of this work is to provide the general solutions of
a class of linear functional equations. Let $n\geq 2$ be an
arbitrarily fixed integer, let further $X$ and $Y$ be linear spaces
over the field $\mathbb{K}$ and let $\alpha_{i}, \beta_{i}\in
\mathbb{K}$, $i=1, \ldots, n$ be arbitrarily fixed constants. We
will describe all those functions $f, f_{i, j}\colon X\times Y\to
\mathbb{K}$, $i, j=1, \ldots, n$ that fulfill functional equation
\[
 f\left(\sum_{i=1}^{n}\alpha_{i}x_{i}, \sum_{i=1}^{n}\beta_{i}y_{i}\right)= \sum_{i, j=1}^{n}f_{i, j}(x_{i}, y_{j})
 \qquad
 \left(x_{i}\in X, y_{i}\in Y, i=1, \ldots, n\right). 
\]
Additionally, necessary and sufficient conditions
will also be given that guarantee the solutions to be non-trivial.
\end{abstract}

\begin{center}
\emph{
 Dedicated to Professor J\'{a}nos Acz\'{e}l on the occasion of his $95$\textsuperscript{th} birthday. 
 }
\end{center}

\section{Introduction}

As J\'{a}nos Acz\'{e}l wrote in his famous and pioneering monograph
\cite{Acz66}: `Functional equations have a long history and occur
almost everywhere. Their influence and applications can be felt in
every field, and all fields benefit from their contact, use, and
technique.' Almost the same can be said about the class of linear
functional equations. This area is one of the most investigated
topic in this field, several authors studied this class, see e.g.
\cite{AczHarMcKSak68, BahOlk17, BahOlk18, ChuEbaNgSahZen95, Dar61,
Kis15, KisLac15, KisVar14, KisVarVin15, Ste99, VarVin09, VarVin15, VinVar15}.

The main purpose of this paper is to describe the general solutions
of a class of linear functional equations. More precisely, we are
interested in the following problem. Let $n\geq 2$ be an arbitrarily
fixed integer, let further $X$ and $Y$ be linear spaces over the
field $\mathbb{K}$ and let $\alpha_{i}, \beta_{i}\in \mathbb{K}$,
$i=1, \ldots, n$ be arbitrarily fixed constants. Assume further that
for the functions $f, f_{i, j}\colon X\times Y\to \mathbb{K}$, $i,
j=1, \ldots, n$,  functional equation
\begin{equation}\label{Eq_main}
 f\left(\sum_{i=1}^{n}\alpha_{i}x_{i}, \sum_{i=1}^{n}\beta_{i}y_{i}\right)= \sum_{i, j=1}^{n}f_{i, j}(x_{i}, y_{j})
 \qquad
 \left(x_{i}\in X, y_{i}\in Y, i=1, \ldots, n\right)
\end{equation}
is fulfilled.

This equation belongs to the class of \emph{linear functional
equations}, that was thoroughly investigated by L.~Sz\'{e}kelyhidi in
\cite{Sze82, Sze89, Sze91}. For the sake of completeness, here we
briefly recall the main results from Sz\'{e}kelyhidi \cite{Sze82}.

\begin{dfn}
If $G, S$ are groups and $n$ is a positive integer, then a function
$A\colon G^{n}\to S$ is said to be \emph{$n$-additive} if it is a
homomorphism in each variable. Let $F\colon G^{n}\to S$ be a
function, then the function $\varphi \colon G\to S$ defined by
\[
 \varphi(x)= F(x, \ldots, x)
 \qquad
 \left(x\in G\right)
\]
is said to be the \emph{diagonal of $F$} and it is denoted by
$\mathrm{diag}(F)$. Further, let
\[
 A_{k}(x, y)=A(\underbrace{x, \ldots, x}_{\text{$k$ times}}, \underbrace{y, \ldots, y}_{\text{$n-k$ times}})
 \qquad 
 \left(x, y\in G\right). 
\]
\end{dfn}

\begin{rem}
Let $G, S$ be groups,  $n$ be a positive integer and 
$A\colon G^{n}\to S$ be an $n$-additive function. Then for all 
$k\in \mathbb{Z}$ and for arbitrary $i\in \left\{1, \ldots, n\right\}$ we have
\[
 A\left(x_{1}, \ldots, x_{i-1}, kx_{i}, x_{i+1}, \ldots, x_{n}\right)
 =
 k A\left(x_{1}, \ldots, x_{i-1}, x_{i}, x_{i+1}, \ldots, x_{n}\right)
 \qquad 
 \left(x_{1}, \ldots, x_{n}\in G\right). 
 \]
\end{rem}

For a function $f$, $\mathrm{rng}(f)$ denotes the range of $f$.

\begin{dfn}
 Let $G, S$ be Abelian groups, let $n$ be a non-negative integer.
 The function $f\colon G\to S$ is said to be \emph{of degree $n$},
 if there exist functions $f_{i}\colon G\to S$ and homomorphisms $\varphi_{i}, \psi_{i}\colon G\to G$ such that
 \[
 \tag{$\mathscr{R}_{1}$}
 \mathrm{rng}(\varphi_{i})\subset \mathrm{rng}(\psi_{i}) 
 \qquad 
 \left(i=1, 2, \ldots, n+1\right)
 \]
 and functional equation
 \begin{equation}\label{Eq_degree_n}
  f(x)+\sum_{i=1}^{n+1}f_{i}\left(\varphi_{i}(x)+\psi_{i}(y)\right)=0
  \qquad
  \left(x, y\in G\right)
 \end{equation}
holds.
\end{dfn}

\begin{dfn}
 Let $G, S$ be Abelian groups, let $n$ be a non-negative integer.
 The function $f\colon G\to S$ is called a  \emph{(generalized) polynomial degree $n$},
 if for all $k=0, 1, \ldots, n$ there exists a $k$-additive mapping
 $A_{k}\colon G^{k}\to S$ such that
 \[
  f=\sum_{k=0}^{n}\mathrm{diag}(A_{k}),
 \]
where $0$-additive functions are to be understood constant
functions.
 \end{dfn}

\begin{thm}[Theorem 3.6 of \cite{Sze82}]\label{T_Szekely1}
 Let $G, S$ be Abelian groups and suppose that $G$ is divisible.
 Let $n$ be a non-negative integer.
 The function $f\colon G\to S$ is of degree $n$ if and only if it is a polynomial of degree $n$.
\end{thm}

\begin{thm}[Theorem 3.9 of \cite{Sze82}]\label{T_Szekely2}
 Let $G, S$ be Abelian groups and suppose that $G$ is divisible and $S$ is torsion free.
 Let $n\in \mathbb{N}$ be a non-negative integer and let
 $\varphi_{i}, \psi_{i}$ be homomorphisms of $G$ onto itself such that
 \[
  \tag{$\mathscr{R}_{2}$} \mathrm{rng}\left(\psi_{j}\circ \psi_{i}^{-1}-\varphi_{j}\circ \varphi_{i}^{-1}\right)= G
  \qquad
  \left(i\neq j, \; i, j=1, \ldots, n+1\right).
 \]
The functions $f_{i}\colon G\to S \; (i=0, 1, \ldots, n+1)$  satisfy
functional equation
\[
 f_{0}(x)+\sum_{i=1}^{n+1}f_{i}\left(\varphi_{i}(x)+\psi_{i}(y)\right)=0
 \qquad
 \left(x, y\in G\right)
\]
if and only if for all $k=0, 1, \ldots, n$ and $i=0, 1, \ldots, n+1$
there exist symmetric $k$-additive functions $A^{(i)}_{k}\colon
G^{k}\to S$ such that
\[
 f_{i}=\sum_{k=0}^{n}\mathrm{diag}\left(A^{(i)}_{k}\right)
 \qquad
 \left(i=0, 1, \ldots, n+1\right)
\]
and the equations
\[
 A^{(0)}_{k, j}(x, 0)+\sum_{i=1}^{n+1}A^{(i)}_{k, j}\left(\varphi_{i}(x), \psi_{i}(y)\right)=0
 \qquad
 \left(x, y\in G\right)
\]
hold for all $j=0, 1, \ldots, n$ and $k=j, j+1, \ldots, n$.
\end{thm}

Observe that equation \eqref{Eq_main} can be reduced to the form \eqref{Eq_degree_n}. 
Indeed, suppose that $n=2$ (or substitute zero in place of the variables except a distinguished pair) and consider the following family of
homomorphisms
\[
 \varphi_{\alpha, \beta}(x, y)=
 \begin{pmatrix}
  \alpha & 0\\
  0& \beta
 \end{pmatrix}
 \cdot
 \begin{pmatrix}
x\\
y
 \end{pmatrix}
\qquad \left(x\in X, y\in Y, \alpha, \beta \in \mathbb{K}\right).
\]
With these notations \eqref{Eq_main} can be re-written as
\begin{multline*}
 f\left(\varphi_{\alpha_{1}, \beta_{1}}(\mathbf{u})+\varphi_{\alpha_{2}, \beta_{2}}(\mathbf{v})\right)
 =
 \\
 f_{1, 1}(\varphi_{1, 1}(\mathbf{u})+\varphi_{0, 0}(\mathbf{v}))+
 f_{1, 2}(\varphi_{1, 0}(\mathbf{u})+\varphi_{0, 1}(\mathbf{v}))+
 f_{2, 1}(\varphi_{0, 1}(\mathbf{u})+\varphi_{1, 0}(\mathbf{v}))+
 f_{2, 2}(\varphi_{0, 0}(\mathbf{u})+\varphi_{1, 1}(\mathbf{v}))
 \\
 \left(\mathbf{u}, \mathbf{v}\in X\times Y\right).
\end{multline*}

At the same time (as it can be seen in the following subsection), we
\emph{cannot} state that the functions involved are polynomials.
This is because the fact that the homomorphisms $\varphi_{\alpha, \beta}$ 
defined above in general do not fulfill range condition
$(\mathscr{R}_{1})$, neither fulfill range condition $(\mathscr{R}_{2})$. 
What is more, they are injective if and only if $\alpha,
\beta \neq 0$ and in such a situation $\varphi_{\alpha, \beta}^{-1}=
\varphi_{\alpha^{-1}, \beta^{-1}}$. Notice that equation
\eqref{Eq_main} involves the projections $\varphi_{1, 0},
\varphi_{0, 1}$ and $\varphi_{0, 0}$. None of these are injective. This
shows that Theorems \ref{T_Szekely1} and \ref{T_Szekely2} cannot be
applied in our situation.

\section{Special cases of the original equation}

\subsection{The one-variable sub-case}

In this sub-case let $n\in \mathbb{N}, n\geq 2$ be arbitrarily
fixed, $X$ be a linear space over the field $\mathbb{K}$ and suppose
that for the functions $f, f_{1}, \ldots, f_{n}\colon X\to
\mathbb{K}$ functional equation
\begin{equation}\label{Eq_one_var}
 f\left(\sum_{i=1}^{n}\alpha_{i}x_{i}\right)=\sum_{i=1}^{n}f_{i}(x_{i})
 \qquad
 \left(x_{1}, \ldots, x_{n}\in X\right)
\end{equation}
holds with certain constants $\alpha_{1}, \ldots, \alpha_{n} \in \mathbb{K}$.

Observe that without loss of generality
\[
\tag{$\ast$}
 f(0)=f_{1}(0)= \ldots =f_{n}(0)=0
\]
can be assumed. Otherwise we consider the functions
\[
 \begin{array}{rcl}
  \widetilde{f}(x)&=&f(x)-f(0)\\
  \widetilde{f_{1}}(x)&=&f_{1}(x)-f_{1}(0)\\
  &\vdots& \\
  \widetilde{f_{n}}(x)&=&f_{n}(x)-f_{n}(0)
 \end{array}
\qquad \left(x\in X\right).
\]
They clearly vanish at zero and they also fulfill the above
functional equation. Therefore from now on we always suppose that $(\ast)$ holds.

As we will see, the solutions of equation \eqref{Eq_one_var} heavily
depend on whether or not there are zeros among the parameters
$\alpha_{1}, \ldots, \alpha_{n}$. We may (and also do) assume that
these parameters are arranged in the following way: there exists  a
non-negative integer $k\leq n$ such that $\alpha_{i}\neq 0$ for $i=1,
\ldots, k$, but $\alpha_{i}=0$ for all $i=k+1, \ldots, n$.

\begin{prop}
 Let $n\in \mathbb{N}, n\geq 2$ be arbitrarily fixed,
$X$ be a linear space over the field $\mathbb{K}$ and suppose that
for the functions $f, f_{1}, \ldots, f_{n}\colon X\to \mathbb{C}$
functional equation \eqref{Eq_one_var} holds with certain constants
$\alpha_{1}, \ldots, \alpha_{n} \in \mathbb{K}$ and assume that
$(\ast)$ is also satisfied. Suppose further that  $\alpha_{i}\neq 0$
for $i=1, \ldots, k$, but $\alpha_{i}=0$ for all $i=k+1, \ldots, n$.
Then
\begin{enumerate}[(i)]
\item in case $k=0$, all the functions $f_{1}, \ldots, f_{n}$ are identically zero and $f\colon X\to \mathbb{C}$ is
any function fulfilling $f(0)=0$, 
\item in case $k=1$, all the functions  $f_{2}, \ldots, f_{n}$ are identically zero and $f, f_{1}\colon X\to \mathbb{C}$ are
any functions vanishing at zero and  fulfilling
\[
 f\left(\alpha_{1}x\right)= f_{1}(x)
 \qquad
 \left(x\in X\right), 
\]
\item otherwise, there exists an additive function $\chi\colon X\to \mathbb{C}$ such that
\[
 f(x)=\chi(x)
 \quad
 \text{and}
 \quad
 f_{i}(x)=\chi(\alpha_{i} x)
 \quad
 \text{for}
 \quad
 i=1, \ldots, k
\]
and the functions $f_{k+1}, \ldots, f_{n}$ are identically zero.
\end{enumerate}
Conversely, the mappings $f, f_{1}, \ldots, f_{n}\colon X\to
\mathbb{C}$ vanish at zero and they also fulfill \eqref{Eq_one_var}.
\end{prop}
\begin{proof}
 In case $k=0$ equation \eqref{Eq_one_var} reduces to
 \[
  \sum_{i=1}^{n}f_{i}(x_{i})=0
  \qquad
  \left(x_{1}, \ldots, x_{n}\in X\right).
 \]
Since we have independent variables, this immediately yields that
the involved functions have to be constant functions. In view of
$(\ast)$ this means that they have to be identically zero and the
only information we get for the function $f$ is that $f(0)=0$.

In case $k\geq 1$, our equation can be written as
\[
 f\left(\sum_{i=1}^{k}\alpha_{i}x_{i}\right)=\sum_{j=1}^{n}f_{j}(x_{j})
 \qquad
 \left(x_{1}, \ldots, x_{n}\in X\right).
\]
With the substitution
\[
 x_{1}=\ldots= x_{k}=0
\]
we obtain that
\[
 0=\sum_{j=k+1}^{n}f_{j}(x_{j})
 \qquad
 \left(x_{k+1}, \ldots, x_{n}\in X\right),
\]
which (similarly as above) yields that the functions $f_{k+1},
\ldots, f_{n}$ are identically zero. Using this, the functions $f,
f_{1}, \ldots, f_{k}$ fulfill
\begin{equation}\label{Eq_one_var_red}
 f\left(\sum_{i=1}^{k}\alpha_{i}x_{i}\right)=\sum_{i=1}^{k}f_{i}(x_{i})
  \qquad
 \left(x_{1}, \ldots, x_{k}\in X\right).
\end{equation}
If $k=1$, this is nothing but
\[
 f\left(\alpha_{1}x\right)= f_{1}(x)
 \qquad
 \left(x\in X\right),
\]
showing that in this case there is nothing to prove.

Assume that $k\geq 2$ and let $i, j\in \left\{1, \ldots, k\right\}$
be different integers. Then equation \eqref{Eq_one_var_red} with
$x_{l}=0$ for $l\in \left\{1, \ldots, k\right\}\setminus \left\{i,
j\right\}$ is
\[
 f(\alpha_{i}x_{i}+\alpha_{j}x_{j})=f_{i}(x_{i})+f_{j}(x_{j})
 \qquad
 \left(x_{i}, x_{j}\in X\right),
\]
which, after introducing the functions
\[
 \widetilde{f_{l}}(x)=f_{l}\left(\dfrac{x}{\alpha_{l}}\right)
 \qquad
 \left(x\in X, l=1, \ldots, k\right)
\]
can be reduced to the system of Pexider equations
\[
 f(x_{i}+x_{j})=\widetilde{f_{i}}(x_{i})+\widetilde{f_{j}}(x_{j})
 \qquad
 \left(x_{i}, x_{j}\in X, i, j\in \left\{1, \ldots, k\right\}, i\neq j\right).
\]
This means that there exists an additive function $\chi\colon X\to
\mathbb{C}$ such that
\[
  f(x)=\chi(x)
 \quad
 \text{and}
 \quad
 f_{i}(x)=\chi(\alpha_{i} x)
 \quad
 \text{for}
 \quad
 i=1, \ldots, k.
\]
\end{proof}

\section{The two-variable case with $n=2$}\label{Sec_n2}

In this section we will focus on functional equation
\begin{multline}\label{Eq_multi2}
 f\left(\alpha_{1}\,x_{1}+\alpha_{2}\,x_{2} , \beta_{1}\,y_{1}+\beta_{2}\,y_{2}\right)
 \\
 =f_{1,1}(x_{1},y_{1})+f_{1,2}(x_{1},y_{2})+f_{2,1}(x_{2},y_{1})+f_{2 ,2}(x_{2},y_{2})
 \qquad
 \left(x_{1}, x_{2}\in X, y_{1}, y_{2}\in Y\right),
\end{multline}
where $f, f_{1, 1}, f_{1, 2}, f_{2, 1}, f_{2, 2}\colon X\times Y\to
\mathbb{K}$ denote the unknown functions and $\alpha_{1},
\alpha_{2}, \beta_{1}, \beta_{2}\in \mathbb{K}$ are given constants.

Observe that without loss of generality
\[
 f(0, 0)=f_{i, j}(0, 0)= 0
\]
can be supposed. Otherwise we consider the functions
\[
 \begin{array}{rcl}
  \widetilde{f}(x, y)&=&f(x, y)-f(0, 0)\\
  \widetilde{f_{i, j}}(x, y)&=&f_{i, j}(x, y)-f_{i, j}(0, 0)
   \end{array}
\qquad \left(x\in X, y\in Y\right).
\]
They clearly vanish at the point $(0, 0)$ and they also fulfill the
same functional equation. Similarly as previously, from now on 
we always suppose that all the involved functions vanish at the point $(0, 0)$.

This section will be divided into two parts. At the first one, we
will consider the so-called degenerate cases, where at least one of
the parameters $\alpha_{1}, \alpha_{2}, \beta_{1}, \beta_{2}$ is
zero. After that the non-degenerate case will follow, that is, when
none of the above parameters are zero.

\subsection{Degenerate cases}

\subsubsection{The homogeneous case $\alpha_{1}=\alpha_{2}=\beta_{1}=\beta_{2}=0$}

In case $\alpha_{1}=\alpha_{2}=\beta_{1}=\beta_{2}=0$ equation
\eqref{Eq_multi2} reduces to
\[
 f_{1,1}(x_{1},y_{1})+f_{1,2}(x_{1},y_{2})+f_{2,1}(x_{2},y_{1})+f_{2 ,2}(x_{2},y_{2})=0
 \qquad
 \left(x_{1}, x_{2}\in X, y_{1}, y_{2}\in Y\right).
\]

\begin{prop}\label{Prop1}
 Let $X$ and $Y$ be linear spaces over the field $\mathbb{K}$ and
 $f_{1, 1}, f_{1, 2}, f_{2, 1}, f_{2, 2}\colon X\times Y\to \mathbb{K}$ be functions such that
 \begin{equation}\label{Eq_deg2}
  f_{1,1}(x_{1},y_{1})+f_{1,2}(x_{1},y_{2})+f_{2,1}(x_{2},y_{1})+f_{2 ,2}(x_{2},y_{2})=0
 \qquad
 \left(x_{1}, x_{2}\in X, y_{1}, y_{2}\in Y\right).
 \end{equation}
 Then and only then for all $i, j=1, 2$ there exist functions $\chi_{i, j}\colon X\to \mathbb{K}$ and
$\zeta_{i, j}\colon Y\to \mathbb{K}$ vanishing at $0$ such that
\[
 f_{i, j}(x, z)= \chi_{i, j}(x)+\zeta_{i, j}(z)
 \qquad
 \left(x\in X, z\in Y, i, j=1, 2\right)
\]
as well as
\[
 \begin{array}{rcl}
\chi_{1,2}(x)+\chi_{1,1}(x)&=&0\\
\chi_{2,2}(x)+\chi_{2,1}(x)&=&0\\
\zeta_{2,1}(z)+\zeta_{1,1}(z)&=&0\\
\zeta_{2,2}(z)+\zeta_{1,2}(z)&=&0
 \end{array}
 \qquad
 \left(x\in X, z\in Y\right).
\]
\end{prop}

\begin{proof}
 For $i, j\in \left\{1, 2\right\}$ let us define the functions $\chi_{i, j}\colon X\to \mathbb{K}$ and $\zeta_{i, j}\colon Y\to \mathbb{K}$ through
 \[
  \chi_{i, j}(x)= f_{i, j}(x, 0)
  \quad
  \text{and}
  \quad
  \zeta_{i, j}(z)= f_{i, j}(0, z)
  \qquad
  \left(x\in X, z\in Y\right).
 \]
With the notation
\[
 E(x_1, x_2, y_1, y_2)
 =f_{1,1}(x_1,y_1)+f_{1,2}(x_1,y_2)+f_{2,1}(x_2,y_1)+f_{2,2}(x_2,y_2)
 \qquad
 \left(x_{1}, x_{2}\in X, y_{1}, y_{2}\in Y\right),
\]
identities
\[
 \begin{array}{rcl}
  E(x, 0, 0, 0)&=&0\\
  E(0, x, 0, 0)&=&0\\
  E(0, 0, z, 0)&=&0\\
  E(0, 0, 0, z)&=&0
 \end{array}
\qquad \left(x\in X, z\in Y\right)
\]
give that
\[
 \begin{array}{rcl}
\chi_{1,2}(x)+\chi_{1,1}(x)&=&0\\
\chi_{2,2}(x)+\chi_{2,1}(x)&=&0\\
\zeta_{2,1}(z)+\zeta_{1,1}(z)&=&0\\
\zeta_{2,2}(z)+\zeta_{1,2}(z)&=&0
 \end{array}
 \qquad
 \left(x\in X, z\in Y\right).
\]
Moreover, equations
\[
 \begin{array}{rcl}
  E(x_1, 0, y_1, 0)&=&0\\
  E(0, x_2, y_1, 0)&=&0\\
  E(x_1, 0, 0, y_2)&=&0\\
  E(0, x_2, 0, y_2)&=&0
 \end{array}
\qquad
 \left(x_{1}, x_{2}\in X, y_{1}, y_{2}\in Y\right)
\]
yield that
\[
 f_{i, j}(x, z)= \chi_{i, j}(x)+\zeta_{i, j}(z)
 \qquad
 \left(x\in X, z\in Y, i, j=1, 2\right),
\]
where we used the previously proved identities, too.
\end{proof}

\subsubsection{The case $\alpha_{1}=\alpha_{2}=\beta_{1}=0$ and $\beta_{2}\neq 0$}

In such a situation  \eqref{Eq_multi2} reduces to
\[
 f(0, \beta_{2}y_{2})= f_{1,1}(x_{1},y_{1})+f_{1,2}(x_{1},y_{2})+f_{2,1}(x_{2},y_{1})+f_{2 ,2}(x_{2},y_{2})
 \qquad
 \left(x_{1}, x_{2}\in X, y_{1}, y_{2}\in Y\right).
\]

Obviously, $\beta_{2}=1$ can be assumed, otherwise we consider the
functions $\widetilde{f_{1, 2}}, \widetilde{f_{2, 2}}\colon X\times
Y\to \mathbb{K}$ defined through
\[
\begin{array}{rcl}
 \widetilde{f_{1, 2}}(x, z)&=&f_{1, 2}\left(x, \frac{z}{\beta_{2}}\right)\\
  \widetilde{f_{2, 2}}(x, z)&=&f_{2, 2}\left(x, \frac{z}{\beta_{2}}\right)
  \end{array}
  \qquad
  \left(x\in X, z\in Y\right).
\]

\begin{prop}\label{Prop3.1.3}
 Let $X$ and $Y$ be linear spaces over the field $\mathbb{K}$ and
 $f, f_{1, 1}, f_{1, 2}, f_{2, 1}, f_{2, 2}\colon X\times Y\to \mathbb{K}$ be functions such that
 \[
  f(0, y_{2})=f_{1,1}(x_{1},y_{1})+f_{1,2}(x_{1},y_{2})+f_{2,1}(x_{2},y_{1})+f_{2 ,2}(x_{2},y_{2})
 \qquad
 \left(x_{1}, x_{2}\in X, y_{1}, y_{2}\in Y\right).
 \]
Then and only then for all $i, j=1, 2$ there exist functions
$\chi_{i, j}\colon X\to \mathbb{K}$ and $\zeta_{i, j}\colon Y\to\mathbb{K}$ vanishing at $0$ such that
\[
 f_{i, j}(x, z)= \chi_{i, j}(x)+\zeta_{i, j}(z)
 \qquad
 \left(x\in X, z\in Y, i, j=1, 2\right)
\]
as well as
\[
 \begin{array}{rcl}
\chi_{1,2}(x)+\chi_{1,1}(x)&=&0\\
\chi_{2,2}(x)+\chi_{2,1}(x)&=&0\\
\zeta_{2,1}(z)+\zeta_{1,1}(z)&=&0\\
\zeta_{2,2}(z)+\zeta_{1,2}(z)&=&f(0, z)
 \end{array}
 \qquad
 \left(x\in X, z\in Y\right).
\]
\end{prop}

\begin{proof}
 For $i, j\in \left\{1, 2\right\}$ let us define the functions $\chi_{i, j}\colon X\to \mathbb{K}$ and $\zeta_{i, j}\colon Y\to \mathbb{K}$ through
 \[
  \chi_{i, j}(x)= f_{i, j}(x, 0)
  \quad
  \text{and}
  \quad
  \zeta_{i, j}(z)= f_{i, j}(0, z)
  \qquad
  \left(x\in X, z\in Y\right).
 \]
Furthermore, let
\begin{multline*}
 E(x_1, x_2, y_1, y_2)
 \\=f(0,y_2)-f_{1,1}(x_1,y_1)-f_{1,2}(x_1,y_2)-f_{2,1}(x_2,y_1)-f_{2,2}(x_2,y_2)
 \qquad
 \left(x_{1}, x_{2}\in X, y_{1}, y_{2}\in Y\right)
\end{multline*}
to obtain the following system of equations
\[
 \begin{array}{rcl}
  E(x_1, 0, y_1, 0)&=&0\\
  E(x_1, 0, 0, y_{2})&=&0\\
  E(0, x_{2}, y_{1}, 0)&=&0\\
  E(0, x_{2}, 0, y_{2})&=&0
 \end{array}
\qquad
 \left(x_{1}, x_{2}\in X, y_{1}, y_{2}\in Y\right),
\]
or equivalently
\[
 \begin{array}{rcl}
  f_{1,1}(x_{1},y_{1})+f_{1,2}(x_{1},0)+f_{2,1}(0, y_{1})&=&0\\
  f_{1,2}(x_{1},y_{2})+f_{1,1}(x_{1},0)+f_{2,2}(0,y_{2})&=&f\left(0, y_{2}\right)\\
  f_{2,1}(x_{2},y_{1})+f_{2,2}(x_{2},0)+f_{1,1}(0, y_{1})&=&0\\
  f_{2,2}(x_{2},y_{2})+f_{2,1}(x_{2},0)+f_{1,2}(0,y_{2})&=&f\left(0  , y_{2}\right)
 \end{array}
 \qquad
 \left(x_{1}, x_{2}\in X, y_{1}, y_{2}\in Y\right).
\]
Finally, the system of equations
\[
 \begin{array}{rcl}
  E(x_{1}, 0, 0, 0)&=&0\\
  E(0, x_{2}, 0, 0)&=&0\\
  E(0, 0, y_{1}, 0)&=&0\\
  E(0, 0, 0, y_{2})&=&0
\end{array}
 \qquad
 \left(x_{1}, x_{2}\in X, y_{1}, y_{2}\in Y\right)
\]
yields that
\[
 \begin{array}{rcl}
  f_{1,2}(x_{1},0)+f_{1,1}(x_{1},0)&=&0\\
  f_{2,2}(x_{2},0)+f_{2,1}(x_{2},0)&=&0\\
  f_{2,1}(0,y_{1})+f_{1,1}(0,y_{1})&=&0\\
  f_{2,2}(0,y_{2})+f_{1,2}(0, y_{2})&=&f\left(0 , y_{2}\right)
   \end{array}
 \qquad
 \left(x_{1}, x_{2}\in X, y_{1}, y_{2}\in Y\right),
\]
which in view of the above definitions completes the proof.
\end{proof}

\subsubsection{The case $\alpha_{1}, \alpha_{2}\neq 0$ and  $\beta_{1}, \beta_{2}=0$}

In such a situation  \eqref{Eq_multi2} implies that 
\[
 f(\alpha_{1}x_{1}+\alpha_{2}x_{2}, 0)= f_{1,1}(x_{1},0)+f_{1,2}(x_{1},0)+f_{2,1}(x_{2},0)+f_{2 ,2}(x_{2},0)
 \qquad
 \left(x_{1}, x_{2}\in X\right)
\]
because the left hand side does not depend on $y_{1}$ and $y_2$.  

Obviously, $\alpha_{1}, \alpha_{2}=1$ can be assumed, otherwise we
consider the functions $\widetilde{f_{1, 2}}, \widetilde{f_{2,
2}}\colon X\times Y\to \mathbb{K}$ defined through
\[
\begin{array}{rcl}
\widetilde{f_{1, 1}}(x, z)&=&f_{1, 1}\left(\frac{x}{\alpha_{1}}, z\right)\\
\widetilde{f_{1, 2}}(x, z)&=&f_{1, 2}\left(\frac{x}{\alpha_{1}}, z\right)\\
 \widetilde{f_{2, 1}}(x, z)&=&f_{2, 1}\left(\frac{x}{\alpha_{2}}, z\right)\\
  \widetilde{f_{2, 2}}(x, z)&=&f_{2, 2}\left(\frac{x}{\alpha_{2}}, z\right)
  \end{array}
  \qquad
  \left(x\in X, z\in Y\right).
\]

\begin{prop}\label{Prop4}
 Let $X$ and $Y$ be linear spaces over the field $\mathbb{K}$ and
 $f, f_{1, 1}, f_{1, 2}, f_{2, 1}, f_{2, 2}\colon X\times Y\to \mathbb{K}$ be functions such that
 \[
  f(x_{1}+x_{2}, 0)=f_{1,1}(x_{1},0)+f_{1,2}(x_{1},0)+f_{2,1}(x_{2},0)+f_{2 ,2}(x_{2},0)
 \qquad
 \left(x_{1}, x_{2}\in X\right).
 \]
 Then and only then there exists an additive function $\chi\colon X\to \mathbb{K}$ such that
 \[
  \begin{array}{rcl}
   f(x, 0)&=&\chi(x)\\
   f_{1, 1}(x, 0)+f_{1, 2}(x, 0)&=&\chi(x)\\
   f_{2, 1}(x, 0)+f_{2, 2}(x, 0)&=&\chi(x)
  \end{array}
\qquad \left(x\in X\right).
\]
\end{prop}

\begin{proof}
Consider the functions $\chi, \varphi, \psi\colon X\to \mathbb{K}$
defined through
\[
 \begin{array}{rcl}
  \chi(x)&=&f(x, 0)\\
  \varphi(x)&=&f_{1, 1}(x, 0)+f_{1, 2}(x, 0)\\
  \psi(x)&=&f_{2, 1}(x, 0)+f_{2, 2}(x, 0)
 \end{array}
\qquad \left(x\in X\right)
\]
to get the following Pexider equation
\[
 \chi(x_{1}+x_{2})=\varphi(x_{1})+\psi(x_{2})
 \qquad
 \left(x_{1}, x_{2}\in X\right).
\]
Since all the functions $\chi, \varphi, \psi$ vanish at zero, we get
that $\varphi\equiv \psi \equiv \chi$ and the function $\chi$ has to
be  additive.
\end{proof}

To finish the discussion of equation \eqref{Eq_multi2} in this special case, apply Proposition \ref{Prop1} to the functions 
\[
\widetilde{f}_{i, j}(x,y)=f_{i, j}(x,y)-f_{i, j}(x,0), 
\qquad 
\left(x\in X, y\in Y\right)
\]
where $f_{i, j}(x,0)$ are given in Proposition \ref{Prop4}.

\subsubsection{The case $\alpha_{1}, \beta_{1}\neq 0$ and  $\alpha_{2}, \beta_{2}=0$}

In such a situation  \eqref{Eq_multi2} implies that 
\[
 f(\alpha_{1}x_{1}, \beta_{1}y_{1})= f_{1,1}(x_{1},y_{1})+f_{1,2}(x_{1},0)+f_{2,1}(0,y_{1})
 \qquad
 \left(x_{1}\in X, y_{1}\in Y\right)
\]
because the left hand side does not depend on $x_2$ and $y_2$. 

Obviously, due to similar reasons as previously, $\alpha_{1},
\beta_{1}=1$ can be assumed. The proof of the following proposition is a straightforward
calculation, so we omit it.

\begin{prop}\label{Prop5}
 Let $X$ and $Y$ be linear spaces over the field $\mathbb{K}$ and
 $f, f_{1, 1}, f_{1, 2}, f_{2, 1}\colon X\times Y\to \mathbb{K}$ be functions.
 Functional equation
 \[
   f(x_{1}, y_{1})= f_{1,1}(x_{1},y_{1})+f_{1,2}(x_{1},0)+f_{2,1}(0,y_{1})
 \qquad
 \left(x_{1}\in X, y_{1}\in Y\right).
 \]
 is fulfilled if and only if there exist functions $\chi\colon X\to \mathbb{K}$ and $\zeta\colon Y\to \mathbb{K}$ such that
 \[
  \begin{array}{rcl}
   f_{1, 2}(x, 0)&=&\chi(x)\\
   f_{2, 1}(0, z)&=&\zeta(z)\\
   f(x, z)-f_{1, 1}(x, z)&=&\chi(x)+\zeta(z)
  \end{array}
\qquad \left(x\in X, z\in Y\right).
 \]
\end{prop}

To finish the discussion of equation \eqref{Eq_multi2} in this special case, apply Proposition \ref{Prop1} 
to the functions 
\[
\begin{array}{rcl}
\widetilde{f}_{1, 1}(x,y)&=&0,\\
\widetilde{f}_{1,2}(x,y)&=&f_{1,2}(x,y)-f_{1,2}(x,0),\\
\widetilde{f}_{2,1}(x,y)&=&f_{2,1}(x,y)-f_{2,1}(0,y),\\
\widetilde{f}_{2, 2}(x,y)&=&f_{2, 2}(x, y), 
\end{array}
\qquad 
\left(x\in X, y\in Y\right), 
\]
where $f_{1,2}(x,0)$ and $f_{2,1}(0,y)$ were determined in Proposition \ref{Prop5}.

\subsubsection{The case $\alpha_{1}, \alpha_{2}, \beta_{1}\neq 0$ and  $\beta_{2}=0$}

In such a situation  \eqref{Eq_multi2} implies that 
\[
 f(\alpha_{1}x_{1}+\alpha_{2}x_{2}, \beta_{1}y_{1})= f_{1,1}(x_{1},y_{1})+f_{1,2}(x_{1},0)+f_{2,1}(x_{2},y_{1})+f_{2, 2}(x_{2}, 0)
 \qquad
 \left(x_{1}\in X, y_{1}\in Y\right), 
\]
because the left hand side does not depend on $y_2$.  

Obviously, due to similar reasons as previously, $\alpha_{1},
\alpha_{2}, \beta_{1}=1$ can be assumed.

\begin{prop}\label{Prop6}
 Let $X$ and $Y$ be linear spaces over the field $\mathbb{K}$ and
 $f, f_{1, 1}, f_{1, 2}, f_{2, 1}, f_{2, 2}\colon X\times Y\to \mathbb{K}$ be functions.
 Functional equation
 \[
   f(x_{1}+x_{2}, y_{1})= f_{1,1}(x_{1},y_{1})+f_{1,2}(x_{1},0)+f_{2,1}(x_{2},y_{1})+f_{2, 2}(x_{2}, 0)
 \qquad
 \left(x_{1}, x_{2}\in X, y_{1}\in Y\right)
 \]
 is fulfilled if and only if there exist a mapping $A\colon X\times Y\to \mathbb{K}$ additive in its first variable 
 and there are functions $\chi, \chi_{1, 1}, \chi_{2, 1}\colon X\to \mathbb{K}$
 and $\zeta, \zeta_{1, 1}, \zeta_{2, 1}\colon Y\to \mathbb{K}$ vanishing at zero so that $\chi$ is additive and 
 \[
 \begin{array}{rcl}
 f(x, z)&=&A(x, z)+\chi(x)+ \zeta(z)\\
 f_{1, 1}(x, z)&=&A(x, z)+\chi_{1, 1}(x)+ \zeta_{1, 1}(z)\\
 f_{2, 1}(x, z)&=&A(x, z)+\chi_{2, 1}(x)+ \zeta_{2, 1}(z)\\
 \end{array}
 \qquad 
 \left(x\in X, z\in Y\right)
\]
and also 
\[
 \begin{array}{rcl}
 \chi(x)&=&f(x, 0)\\
  \chi_{1, 1}(x)&=&f_{1, 1}(x, 0)\\
  \chi_{2, 1}(x)&=&f_{2, 1}(x, 0)\\
 \zeta(z)&=&\zeta_{1, 1}(z)+\zeta_{2, 1}(z)\\
   f_{1, 2}(x, 0)&=& \chi(x)-\chi_{1, 1}(x)\\
   f_{2, 2}(x, 0)&=& \chi(x)-\chi_{2, 1}(x) 
 \end{array}
\qquad 
\left(x\in X, z\in Y\right)
\]
hold. 
\end{prop}

\begin{proof}
 With the substitution $y_{1}=0$ our equation yields that 
 \[
   f(x_{1}+x_{2}, 0)= f_{1,1}(x_{1}, 0)+f_{1,2}(x_{1},0)+f_{2,1}(x_{2}, 0)+f_{2, 2}(x_{2}, 0)
 \qquad
 \left(x_{1}, x_{2}\in X, \right).
\]
From this we immediately get that 
\[
 \widetilde{f}(x_{1}+x_{2}, y_{1})= \widetilde{f}_{1, 1}(x_{1}, y_{1})+ \widetilde{f}_{2, 1}(x_{2}, y_{1})
 \qquad
 \left(x_{1}, x_{2}\in X, y_{1}\in Y\right), 
\]
where the functions $\widetilde{f}, \widetilde{f}_{1, 1}, \widetilde{f}_{2, 1}\colon X\times Y\to \mathbb{K}$ are defined by 
\[
 \begin{array}{rcl}
  \widetilde{f}(x, y)&=&f(x, y)-f(x, 0)\\
   \widetilde{f}_{1, 1}(x, y)&=&f_{1, 1}(x, y)-f_{1, 1}(x, 0)\\
    \widetilde{f}_{2, 1}(x, y)&=&f_{2, 1}(x, y)-f_{2, 1}(x, 0)\\
 \end{array}
\qquad 
\left(x\in X, y\in Y\right). 
\]
This means that the functions $\widetilde{f}, \widetilde{f}_{1, 1}, \widetilde{f}_{2, 1}$ fulfill a Pexider equation on 
$X$ for any fixed $y\in Y$. Thus there exists a mapping $A\colon X\times Y\to \mathbb{K}$ additive in its first variable and 
functions $\zeta, \zeta_{1, 1}, \zeta_{2, 1}\colon Y\to \mathbb{K}$ so that 
\[
 \zeta(z)=\zeta_{1, 1}(z)+\zeta_{2, 1}(z)
 \qquad 
 \left(z\in Y\right)
\]
and 
\[
 \begin{array}{rcl}
 \widetilde{f}(x, z)&=&A(x, z)+ \zeta(z)\\
 \widetilde{f}_{1, 1}(x, z)&=&A(x, z)+ \zeta_{1, 1}(z)\\
 \widetilde{f}_{2, 1}(x, z)&=&A(x, z)+ \zeta_{2, 2}(z)\\
 \end{array}
 \qquad 
 \left(x\in X, z\in Y\right). 
\]
In terms of the functions $f, f_{1, 1}, f_{2, 1}$ this means that 
\[
 \begin{array}{rcl}
 f(x, z)&=&A(x, z)+\chi(x)+ \zeta(z)\\
 f_{1, 1}(x, z)&=&A(x, z)+\chi_{1, 1}(x)+ \zeta_{1, 1}(z)\\
 f_{2, 1}(x, z)&=&A(x, z)+\chi_{2, 1}(x)+ \zeta_{2, 2}(z)\\
 \end{array}
 \qquad 
 \left(x\in X, z\in Y\right), 
\]
where 
\[
 \begin{array}{rcl}
  \chi(x)&=&f(x, 0)\\
  \chi_{1, 1}(x)&=&f_{1, 1}(x, 0)\\
  \chi_{2, 1}(x)&=&f_{2, 1}(x, 0)
 \end{array}
\qquad
\left(x\in X\right). 
\]
Observe that $\chi$ is additive. 
Indeed, using the above the forms of $f, f_{1, 1}$ and $f_{2, 2}$, our equation with $y_{1}=0$ 
and the fact that $A$ is additive in its first variable, 
we obtain that 
\[
 \chi(x_{1}+x_{2})
 = \chi_{1, 1}(x_{1})+f_{1, 2}(x_{1}, 0)+ \chi_{2, 1}(x_2)+ f_{2, 2}(x_{2}, 0)
 \qquad 
 \left(x_{1}, x_{2}\in X\right), 
\]
that is, $\chi$ fulfills a Pexider equation. Since $\chi(0)=0$, this means that $\chi$ has to be additive. 
Thus, using again the form of the functions $f, f_{1, 1}, f_{2, 1}$ and our equation with $x_{2}=0$,  we get that 
\[
 f_{1, 2}(x, 0)= \chi(x)-\chi_{1, 1}(x) 
 \qquad 
 \left(x\in X\right). 
\]
Similarly, our equation with $x_{1}=0$ implies that 
\[
 f_{2, 2}(x, 0)= \chi(x)-\chi_{2, 1}(x) 
 \qquad 
 \left(x\in X\right). 
\]

\end{proof}

To finish the discussion of equation \eqref{Eq_multi2} in this special case, 
apply Proposition \ref{Prop1} to the functions
\[
\begin{array}{rcl}
\widetilde{f}_{1, 1}(x,y)&=&0,\\
\widetilde{f}_{1,2}(x,y)&=&f_{1,2}(x,y)-f_{1,2}(x,0),\\
\widetilde{f}_{2,1}(x,y)&=&0,\\
\widetilde{f}_{2, 2}(x,y)&=&f_{2,2}(x,y)-f_{2,2}(x, 0), 
\end{array}
\qquad 
\left(x\in X, y\in Y\right)
\]
where $f_{1,2}(x,0)$ and $f_{2,2}(x, 0)$ are given in Proposition \ref{Prop6}.

\subsection{The non-degenerate case}\label{Sec_nondeg}

After making clear the degenerate cases, now we can focus on the
case $\alpha_{1}, \alpha_{2}, \beta_{1}, \beta_{2}\neq 0$ and
provide the general solution of the functional equation
\begin{multline*}
 f\left(\alpha_{1}\,x_{1}+\alpha_{2}\,x_{2} , \beta_{1}\,y_{1}+\beta_{2}\,y_{2}\right)
 \\
 =f_{1,1}(x_{1},y_{1})+f_{1,2}(x_{1},y_{2})+f_{2,1}(x_{2},y_{1})+f_{2 ,2}(x_{2},y_{2})
 \qquad
 \left(x_{1}, x_{2}\in X, y_{1}, y_{2}\in Y\right),
\end{multline*}
where $f, f_{1, 1}, f_{1, 2}, f_{2, 1}, f_{2, 2}\colon X\times Y\to
\mathbb{K}$ denote the unknown functions and $\alpha_{1},
\alpha_{2}, \beta_{1}, \beta_{2}\in \mathbb{K}$ are given constants.

Obviously, it is enough to consider the case
$\alpha_{1}=\alpha_{2}=\beta_{1}=\beta_{2}=1$, that is, to consider
the following functional equation
\[
  f\left(x_{1}+x_{2} , y_{1}+y_{2}\right)
 =f_{1,1}(x_{1},y_{1})+f_{1,2}(x_{1},y_{2})+f_{2,1}(x_{2},y_{1})+f_{2 ,2}(x_{2},y_{2})
 \qquad
 \left(x_{1}, x_{2}\in X, y_{1}, y_{2}\in Y\right).
\]

In this subsection we always assume that the characteristic of the field $\mathbb{K}$ is
\emph{different} from $2$. 
 
\begin{prop}\label{Prop_nondeg}
Let $X$ and $Y$ be linear spaces over the field $\mathbb{K}$. Then
the functions $f, f_{1, 1}, f_{1, 2}, f_{2, 1}, \allowbreak f_{2,
2}\colon \allowbreak X\times Y\to \mathbb{K}$ satisfy functional
equation
\begin{equation}\label{Eq_non-deg}
  f\left(x_{1}+x_{2} , y_{1}+y_{2}\right)
 =f_{1,1}(x_{1},y_{1})+f_{1,2}(x_{1},y_{2})+f_{2,1}(x_{2},y_{1})+f_{2 ,2}(x_{2},y_{2})
 \qquad
 \left(x_{1}, x_{2}\in X, y_{1}, y_{2}\in Y\right).
\end{equation}
 if and only if
 \[
  \begin{array}{rcl}
   f(x, z)&=&A(x, z)+\chi(x)+\zeta(z)\\
    f_{i, j}(x, z)&=&A(x, z)+\chi_{i, j}(x)+\zeta_{i, j}(z)
  \end{array}
\qquad \left(x\in X, z\in Z\right),
 \]
where the mapping $A\colon X\times Y\to \mathbb{K}$ is a
bi-additive function and for $i, j\in \left\{1, 2\right\}$ $\chi,
\chi_{i, j}\colon X\to \mathbb{K}$ as well as $\zeta, \zeta_{i,
j}\colon Y\to \mathbb{K}$ are functions such that $\chi$ and $\zeta$
are additive functions and $\chi_{i,j}$ and $\zeta_{i,j}$ vanish at the point $(0, 0)$ and 
\[
 \begin{array}{rcl}
  \chi(x)&=&\chi_{1, 1}(x)+\chi_{1, 2}(x)=\chi_{2, 1}(x)+\chi_{2, 2}(x)\\
  \zeta(z)&=& \zeta_{1, 1}(z)+\zeta_{2, 1}(z)=\zeta_{1, 2}(z)+\zeta_{2, 2}(z)
 \end{array}
\qquad \left(x\in X, z\in Y\right)
\]
are also fulfilled.
 \end{prop}
 \begin{proof}
  Assume that the functions $f, f_{1, 1}, f_{1, 2}, f_{2, 1}, \allowbreak f_{2, 2}\colon \allowbreak X\times Y\to \mathbb{K}$
  fulfill functional equation \eqref{Eq_non-deg} for any $x_{1}, x_{2}\in X$ and $y_{1}, y_{2}\in Y$.
  With the substitution $y_{2}=0$  we obtain that
  \[
   f\left(x_{2}+x_{1} , y_{1}\right)=
   f_{2,1}( x_{2},y_{1})+f_{2,2}(x_{2},0)+f_{1,1}(x_{1}, y_{1})+f_{1,2}(x_{1},0)
   \qquad
   \left(x_{1}, x_{2}\in X, y_{1}\in Y\right),
  \]
which immediately implies that
\[
 f\left(x_{2}+x_{1} , y_{1}\right)=
  \widetilde{f_{2,1}}( x_{2},y_{1})+\widetilde{f_{1,1}}(x_{1}, y_{1})
  \qquad
   \left(x_{1}, x_{2}\in X, y_{1}\in Y\right),
\]
where the functions $\widetilde{f_{1, 1}}, \widetilde{f_{2, 1}}$ are
defined by
\[
 \begin{array}{rcl}
  \widetilde{f_{1, 1}}(x, z)&=& f_{1, 1}(x, z)+f_{1, 2}(x, 0)\\
  \widetilde{f_{2, 1}}(x, z)&=&f_{2, 1}(x, z)+f_{2, 2}(x, 0)
 \end{array}
\qquad \left(x\in X, z\in Y\right).
\]
This means that there exists a function $A^{(1)}\colon X\times Y\to
\mathbb{K}$ which is additive in its first variable and a function
$\zeta\colon Y\to \mathbb{K}$ vanishing at zero such that
\[
 f(x, z)= A^{(1)}(x, z)+\zeta(z)
 \qquad
 \left(x\in X, z\in Y\right).
\]
Substituting this form into equation \eqref{Eq_non-deg}, with
$x_{2}=0$ and with a similar argument we receive that
\[
 A^{(1)}(x, z)= A(x, z)+\chi(x)
 \qquad
 \left(x\in X, z\in Y\right),
\]
where $A\colon X\times Y\to \mathbb{K}$ is a bi-additive mapping and
$\chi\colon X\to \mathbb{K}$ is a function that vanishes at zero.

All in all this means that
\[
 f(x, z)= A(x, z)+\chi(x)+\zeta(z)
 \qquad
 \left(x\in X, z\in Y\right).
\]
Additionally, equation \eqref{Eq_non-deg}, first with
$y_{1}=y_{2}=0$ yields that $\chi$ has to be additive and secondly, 
with $x_{1}=x_{2}=0$ we receive that the function $\zeta$ also has
to be additive, too.

Define the functions $F_{1, 1}, F_{1, 2}, F_{2, 1}, F_{2, 2}$ on
$X\times Y$ through
\[
 \begin{array}{rcl}
  F_{1, 1}(x, z)&=&f_{1, 1}(x, z)-A(x, z)-\dfrac{\chi(x)}{2}-\dfrac{\zeta(z)}{2}\\[3mm]
  F_{1, 2}(x, z)&=&f_{1, 2}(x, z)-A(x, z)-\dfrac{\chi(x)}{2}-\dfrac{\zeta(z)}{2}\\[3mm]
  F_{2, 1}(x, z)&=&f_{2, 1}(x, z)-A(x, z)-\dfrac{\chi(x)}{2}-\dfrac{\zeta(z)}{2}\\[3mm]
   F_{2, 2}(x, z)&=&f_{2, 2}(x, z)-A(x, z)-\dfrac{\chi(x)}{2}-\dfrac{\zeta(z)}{2}\\
 \end{array}
 \qquad
 \left(x\in X, z\in Y\right)
\]
to deduce that they fulfill functional equation \eqref{Eq_deg2}. Due
to Proposition \ref{Prop1}, for all $i, j=1, 2$ there exist
functions $\widetilde{\chi}_{i, j}\colon X\to \mathbb{K}$ and $\widetilde{\zeta}_{i,j}\colon Y\to \mathbb{K}$ vanishing at zero such that
\[
 F_{i, j}(x, z)= \widetilde{\chi}_{i, j}(x)+\widetilde{\zeta}_{i, j}(z)
 \qquad
 \left(x\in X, z\in Y, i, j=1, 2\right),
\]
that is, for the functions $f_{i,j}$ we have
\[
  f_{i, j}(x, z)=A(x, z)+\chi_{i, j}(x)+\zeta_{i, j}(z)
\qquad \left(i, j\in \left\{1, 2\right\}, x\in X, z\in Y\right).
\]
Finally, using the equations in Proposition \ref{Prop1} 
for the functions $\widetilde{\chi}_{i, j}$ and $\widetilde{\zeta}_{i, j}$, identities
\[
 \begin{array}{rcl}
  f(x, 0)&=&\chi(x)= f_{1, 1}(x, 0)+f_{1, 2}(x, 0)=\chi_{1, 1}(x)+\chi_{1, 2}(x)\\
  f(x, 0)&=&\chi(x)=f_{2, 1}(x, 0)+f_{2, 2}(x, 0)=\chi_{2, 1}(x)+\chi_{2, 2}(x)\\
  f(0, z)&=&\zeta(z)=f_{1, 1}(0, z)+f_{2, 1}(0, z)=\zeta_{1, 1}(z)+\zeta_{2, 1}(z)\\
  f(0, z)&=&\zeta(z)= f_{1, 2}(0, z)+f_{2, 2}(0, z)=\zeta_{1, 2}(z)+\zeta_{2, 2}(z)
 \end{array}
\qquad \left(x\in X, z\in Y\right)
\]
complete the proof.
 \end{proof}

\subsection{Related equations}

\subsubsection{The functional equation of bi-additivity}

As a trivial consequence of the results of the previous section we
get the following.

\begin{cor}
 Let $X$ and $Y$ be linear spaces over the field $\mathbb{K}$ with $\mathrm{char}(\mathbb{K})\neq 2$. 
 The mapping $f\colon X\times Y\to \mathbb{K}$ fulfills the functional equation of bi-additivity, that is,
 \[
  f\left(x_{1}+x_{2} , y_{1}+y_{2}\right)
 \\
 =f(x_{1},y_{1})+f(x_{1},y_{2})+f(x_{2},y_{1})+f(x_{2},y_{2})
 \qquad
 \left(x_{1}, x_{2}\in X, y_{1}, y_{2}\in Y\right)
 \]
 if and only if $f$ is bi-additive.
\end{cor}

\subsubsection{The rectangle equation}

Let $X$ and $Y$ be linear spaces over the field $\mathbb{K}$ and 
let $f\colon X\times Y\to \mathbb{K}$ be a function.

Then functional equation
\[
 f(x+u, y+v)+f(x+u, y-v)+f(x-u, y+v)+f(x-u, y-v)= 4f(x, y)
 \qquad
 \left(x, y\in X, u, v\in Y\right),
\]
or equivalently (provided that $\mathrm{char}(\mathbb{K})\neq 2$)
\[
  4f\left(\frac{x_{1}+x_{2}}{2} , \frac{y_{1}+y_{2}}{2}\right)
 =f(x_{1},y_{1})+f(x_{1},y_{2})+f(x_{2},y_{1})+f(x_{2},y_{2})
 \qquad
 \left(x_{1}, x_{2}\in X, y_{1}, y_{2}\in Y\right).
\]
is called the \emph{rectangle equation}.

Indeed, both the above equations express the following: the value of
$f$ at the center of any rectangle, with parallel sides to the coordinate
axes, equals the mean of the values of $f$ at the vertices.

\begin{center}
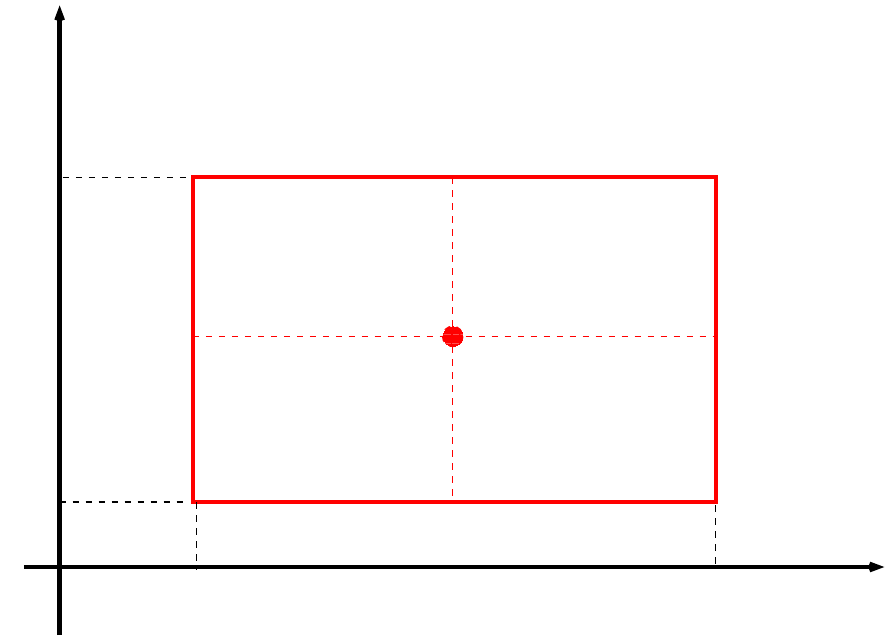
\end{center}

This equation as well as its generalization were investigated (among
others) in \cite{AczHarMcKSak68, ChuEbaNgSahZen95, Ste99}.

With the aid of the results of the previous section, we obtain the
following straightaway.

\begin{prop}\label{T_rectangle}
 Let $X$ and $Y$ be linear spaces over the field $\mathbb{K}$ with  $\mathrm{char}(\mathbb{K})\neq 2$ and
$f\colon X\times Y\to \mathbb{K}$ be a function. The function
$f$  fulfills the rectangle equation, i.e.,
\[
  4f\left(\frac{x_{1}+x_{2}}{2} , \frac{y_{1}+y_{2}}{2}\right)
 =f(x_{1},y_{1})+f(x_{1},y_{2})+f(x_{2},y_{1})+f(x_{2},y_{2})
 \qquad
 \left(x_{1}, x_{2}\in X, y_{1}, y_{2}\in Y\right),
\]
if and only if there exists a bi-additive mapping $A\colon X\times
Y\to \mathbb{K}$ and additive functions $\chi\colon X\to \mathbb{K}$
and $\zeta\colon Y\to \mathbb{K}$  such that
\[
 f(x, z)= A(x, z)+\chi(x)+\zeta(z)
 \qquad
 \left(x\in X, z\in Y\right).
\]
\end{prop}

\begin{rem}
 In case $\mathrm{char}(\mathbb{K})=2$, the rectangle equation reduces to 
 equation
 \[
 f(x+u, y+v)+f(x+u, y-v)+f(x-u, y+v)+f(x-u, y-v)= 0
 \qquad
 \left(x, y\in X, u, v\in Y\right),
\]
or equivalently 
\[
 f(x_{1},y_{1})+f(x_{1},y_{2})+f(x_{2},y_{1})+f(x_{2},y_{2})=0
 \qquad
 \left(x_{1}, x_{2}\in X, y_{1}, y_{2}\in Y\right).
\]
Thus, Proposition \ref{Prop1} yields that there exist functions 
$\chi\colon X\to \mathbb{K}$ and $\zeta\colon Y\to \mathbb{K}$ such that 
\[
 f(x, y)= \chi(x)+\zeta(y) 
 \qquad 
 \left(x\in X, y\in Y\right). 
\]
\end{rem}

\subsubsection*{The Cauchy equation on $X\times Y$}

\begin{prop}
 Let $X$ and $Y$ be linear spaces over the field $\mathbb{K}$ and
$f, g, h\colon X\times Y\to \mathbb{K}$ be functions. Then
functional equation
\begin{equation}
 f(x_{1}+x_{2}, y_{1}+y_{2})= g(x_{1}, y_{1})+h(x_{2}, y_{2})
 \qquad
 \left(x_{1}, x_{2}\in X, y_{1}, y_{2}\in Y\right)
\end{equation}
holds if and only if there exist additive functions $\chi \colon
X\to \mathbb{K}$, $\zeta\colon Y\to \mathbb{K}$ 
such that
\[
 \begin{array}{rcl}
  f(x, y)&=&\chi(x)+\zeta(z) \\
  g(x, y)&=&\chi(x)+\zeta(z) \\
  h(x, y)&=&\chi(x)+\zeta(z) \\
 \end{array}
\qquad \left(x\in X, z\in Y\right).
\]
\end{prop}

\section{On the reduction of equations with $n>2$ to  the two-variable case  }

In this section we intend to investigate the following problem. Let
$X$ and $Y$ be linear spaces over the field $\mathbb{K}$, let
further $\alpha_{i}, \beta_{i}\in \mathbb{K}$, $i=1, \ldots, n$ be
arbitrarily fixed constants. Assume further that for the functions
$f, f_{i, j}\colon X\times Y\to \mathbb{K}$, $i, j=1, \ldots, n$,
functional equation
\begin{equation}\label{Eq_n}
 f\left(\sum_{i=1}^{n}\alpha_{i}x_{i}, \sum_{i=1}^{n}\beta_{i}y_{i}\right)= \sum_{i, j=1}^{n}f_{i, j}(x_{i}, y_{j})
 \qquad
 \left(x_{i}\in X, y_{i}\in Y, i=1, \ldots, n\right)
\end{equation}
is fulfilled.

We will show that in case $n>2$, the results of the previous section can be applied. 
Indeed, let $\lambda, \kappa, \mu, \nu \in \left\{1, \ldots,
n\right\}$ such that $\lambda\neq \kappa$ and $\mu\neq \nu$, but
otherwise arbitrary. In this case equation \eqref{Eq_n} with the
substitutions
\[
 x_{i}=0 \text{ if } i\neq \lambda, \kappa \qquad \text{ and } \qquad 
 y_{j}=0 \text{ if } j\neq \mu, \nu
\]
yields that
\begin{multline*}
 f\left(\alpha_{\lambda}x_{\lambda}+\alpha_{\kappa}x_{\kappa}, \beta_{\mu}y_{\mu}+\beta_{\nu}y_{\nu}\right)
 =
 f_{\lambda, \mu}(x_{\lambda}, y_{\mu})
 +f_{\lambda, \nu}(x_{\lambda}, y_{\nu})
 +f_{\kappa, \mu}(x_{\kappa}, y_{\mu})
 +f_{\kappa, \nu}(x_{\kappa}, y_{\nu})\\
 +\sum_{j\neq \mu, \nu}f_{\lambda, j}(x_{\lambda}, 0)
 +\sum_{j\neq \mu, \nu}f_{\kappa, j}(x_{\kappa}, 0)
 +\sum_{i\neq \lambda, \kappa}f_{i, \mu}(0, y_{\mu})
 +\sum_{i\neq \lambda, \kappa}f_{i, \nu}(0, y_{\nu})
\end{multline*}
for any $x_{\lambda}, x_{\kappa}\in X$ and $y_{\mu}, y_{\nu}\in Y$.
Consider the functions $\widetilde{f_{\lambda, \mu}},
\widetilde{f_{\kappa, \nu}}\colon X\times Y\to \mathbb{K}$ defined
by
\[
 \widetilde{f_{\lambda, \mu}}(x, z)=f_{\lambda, \mu}(x, z)+\sum_{j\neq \mu, \nu}f_{\lambda, j}(x, 0)+\sum_{i\neq \lambda, \kappa}f_{i, \mu}(0, z)
 \qquad
 \left(x\in X, z\in Y\right)
\]
and
\[
 \widetilde{f_{\kappa, \nu}}(x, z)=f_{\kappa, \nu}(x, z)+\sum_{j\neq \mu, \nu}f_{\kappa, j}(x, 0)+\sum_{i\neq \lambda, \kappa}f_{i, \nu}(0, z)
 \qquad
 \left(x\in X, z\in Y\right)
\]
to receive that
\begin{equation}
  f\left(\alpha_{\lambda}x_{\lambda}+\alpha_{\kappa}x_{\kappa}, \beta_{\mu}y_{\mu}+\beta_{\nu}y_{\nu}\right)
 =
 \widetilde{f_{\lambda, \mu}}(x_{\lambda}, y_{\mu})
 +f_{\lambda, \nu}(x_{\lambda}, y_{\nu})
 +f_{\kappa, \mu}(x_{\kappa}, y_{\mu})
 +\widetilde{f_{\kappa, \nu}}(x_{\kappa}, y_{\nu})
\end{equation}
is satisfied for any $x_{\lambda}, x_{\kappa}\in X$ and $y_{\mu},
y_{\nu}\in Y$. This equation can however be handled with the aid of
the results of Section \ref{Sec_n2}.

\section{The case of a single unknown function in the equation --- existence of non-trivial solutions}

Let $X$ and $Y$ be linear spaces over the same field $\mathbb{K}$
and consider the following functional equation
\begin{multline}\label{Eq_non-trivial}
 f\left(\alpha_{1}\,x_{1}+\alpha_{2}\,x_{2} , \beta_{1}\,y_{1}+\beta_{2}\,y_{2}\right)
 \\
 =\gamma_{1, 1}f(x_{1},y_{1})+\gamma_{1, 2}f(x_{1},y_{2})+\gamma_{2, 1}f(x_{2},y_{1})+\gamma_{2, 2}f(x_{2},y_{2})
 \qquad
 \left(x_{1}, x_{2}\in X, y_{1}, y_{2}\in Y\right),
\end{multline}
where $f\colon X\times Y\to \mathbb{K}$ denotes the unknown function
and $\alpha_{1}, \alpha_{2}, \beta_{1}, \beta_{2}\in \mathbb{K}$ and
$\gamma_{1, 1}, \gamma_{1, 2}, \gamma_{2, 1}, \gamma_{2, 2}\in
\mathbb{K}$ are given constants. 

Recall that due to the linearity of the above equation we may (and we also do) suppose that 
\[
 f(0, 0)=0
\]
holds. Otherwise the function
\[
\widetilde{f}(x, y)=f(x, y)-f(0, 0)
\qquad \left(x\in X, y\in Y\right)
\]
can be considered. This function clearly vanishes at the point $(0, 0)$ and it fulfills the
same functional equation, too.

Furthermore, the linearity of the investigated equation implies 
that the identically zero function is always a solution. 
In this section we would like to study under what conditions
admits equation \eqref{Eq_non-trivial} a \emph{non-identically zero}
solution. Clearly, in every case the results of the previous
sections can be applied with the choice
\[
 f_{i, j}(x, y)= \gamma_{i, j}f(x, y)
 \qquad
 \left(x\in X, y\in Y\right).
\]
This means that the assumption that the function $f$ is not
identically zero will imply algebraic conditions for the involved
parameters $\alpha_{1}, \alpha_{2}, \beta_{1}, \beta_{2}\in
\mathbb{K}$ and $\gamma_{1, 1}, \gamma_{1, 2}, \gamma_{2, 1},
\gamma_{2, 2}\in \mathbb{K}$.

Similarly as before, first we consider the so-called degenerate
cases.

\subsection{Degenerate cases}

\subsubsection{The case $\alpha_{1}=\alpha_{2}=\beta_{1}=\beta_{2}=0$}

In case $\alpha_{1}=\alpha_{2}=\beta_{1}=\beta_{2}=0$ equation
\eqref{Eq_non-trivial} reduces to
\[
 \gamma_{1, 1}f(x_{1},y_{1})+\gamma_{1, 2}f(x_{1},y_{2})+\gamma_{2, 1}f(x_{2},y_{1})+\gamma_{2, 2}f(x_{2},y_{2})=0
 \qquad
 \left(x_{1}, x_{2}\in X, y_{1}, y_{2}\in Y\right),
\]
where $\gamma_{i, j}\in \mathbb{K}$ for any $i, j\in \left\{1,
2\right\}$.

\begin{prop}\label{Prop_nontriv1}
 Let $X$ and $Y$ be linear spaces over the field $\mathbb{K}$, $\gamma_{i, j}\in \mathbb{K}$ be given constants such that not all of them are zero and
 $f\colon X\times Y\to \mathbb{K}$ be a function such that
 \begin{equation}\label{Eq_nontriv_deg2}
   \gamma_{1, 1}f(x_{1},y_{1})+\gamma_{1, 2}f(x_{1},y_{2})+\gamma_{2, 1}f(x_{2},y_{1})+\gamma_{2, 2}f(x_{2},y_{2})=0
 \qquad
 \left(x_{1}, x_{2}\in X, y_{1}, y_{2}\in Y\right).
 \end{equation}
 Then and only then there exist functions $\chi\colon X\to \mathbb{K}$ and
$\zeta\colon Y\to \mathbb{K}$ vanishing at zero such that
\[
 f(x, z)= \chi(x)+\zeta(z)
 \qquad
 \left(x\in X, z\in Y\right).
\]
Furthermore
\begin{enumerate}[(i)]
 \item either the following system of linear equations
 \[
  \begin{array}{rcl}
   \gamma_{1, 1}+\gamma_{1, 2}&=&0\\
   \gamma_{2, 1}+\gamma_{2, 2}&=&0
  \end{array}
 \]
is fulfilled or the function $\chi$ is identically zero.
\item either the following system of linear equations
 \[
  \begin{array}{rcl}
   \gamma_{1, 1}+\gamma_{2, 1}&=&0\\
   \gamma_{1, 2}+\gamma_{2, 2}&=&0
  \end{array}
 \]
is fulfilled or the function $\zeta$ is identically zero.
\end{enumerate}
\end{prop}
\begin{proof}
 In view of Proposition \ref{Prop1} we get that there exist functions 
 $\chi\colon X\to \mathbb{K}$ and $\zeta\colon Y\to \mathbb{K}$ vanishing at zero such that 
 \[
 f(x, z)= \chi(x)+\zeta(z)
 \qquad
 \left(x\in X, z\in Y\right).
\]
Using this representation of the function $f$, equation \eqref{Eq_nontriv_deg2} yields that 
\begin{multline*}
\gamma_{1, 1}\left(\chi(x_{1})+ \zeta(y_{1})\right)+\gamma_{1, 2}\left(\chi(x_{1})+\zeta(y_{2})\right)
+\gamma_{2, 1}\left(\chi((x_{2})+\zeta(y_{1})\right)+\gamma_{2, 2}\left(\chi(x_{2})+\zeta(y_{2})\right)=0
 \\
 \left(x_{1}, x_{2}\in X, y_{1}, y_{2}\in Y\right), 
 \end{multline*}
 or equivalently
 \begin{multline*}
 \chi(x_{1})\left(\gamma_{1, 1}+\gamma_{1, 2}\right)+
 \chi(x_{2})\left(\gamma_{2, 1}+\gamma_{2, 2}\right)+
 \zeta(y_{1})\left(\gamma_{1, 1}+\gamma_{2, 1}\right)+
 \zeta(y_{2})\left(\gamma_{1, 2}+\gamma_{2, 2}\right)
 =0
  \\
 \left(x_{1}, x_{2}\in X, y_{1}, y_{2}\in Y\right). 
 \end{multline*}
Since we have independent variables we get that 
\[
  \begin{array}{rcl}
   \gamma_{1, 1}+\gamma_{1, 2}&=&0\\
   \gamma_{2, 1}+\gamma_{2, 2}&=&0
  \end{array}
 \]
is fulfilled or the function $\chi$ is identically zero.  
Similarly,  
\[
  \begin{array}{rcl}
   \gamma_{1, 1}+\gamma_{2, 1}&=&0\\
   \gamma_{1, 2}+\gamma_{2, 2}&=&0
  \end{array}
 \]
holds or the function $\zeta$ is identically zero.
\end{proof}

\subsubsection{The case $\alpha_{1}=\alpha_{2}=\beta_{1}=0$ and $\beta_{2}\neq 0$}

In such a situation  \eqref{Eq_non-trivial} reduces to
\[
 f(0, \beta_{2}y_{2})=  \gamma_{1, 1}f(x_{1},y_{1})+\gamma_{1, 2}f(x_{1},y_{2})+\gamma_{2, 1}f(x_{2},y_{1})+\gamma_{2, 2}f(x_{2},y_{2})
 \qquad
 \left(x_{1}, x_{2}\in X, y_{1}, y_{2}\in Y\right).
\]
In view of Proposition \ref{Prop3.1.3}, the proof of the following proposition is straightforward and similar to that of Proposition 
\ref{Prop_nontriv1}. 
The basic step is to consider $f$ as the sum of single variable functions (Proposition \ref{Prop3.1.3}) 
and substitute such a special form of $f$ into the functional equation. 

\begin{prop}\label{Prop_nontriv2}
 Let $X$ and $Y$ be linear spaces over the field $\mathbb{K}$, $\beta_{2}, \gamma_{i, j}\in \mathbb{K}$ 
 be given constants such that not all of them are zero,  and
 $f\colon X\times Y\to \mathbb{K}$ be a function such that
 \begin{equation}\label{Eq_nontriv_deg3}
  f(0, \beta_{2}y_{2})= \gamma_{1, 1}f(x_{1},y_{1})+\gamma_{1, 2}f(x_{1},y_{2})+\gamma_{2, 1}f(x_{2},y_{1})+\gamma_{2, 2}f(x_{2},y_{2})
 \qquad
 \left(x_{1}, x_{2}\in X, y_{1}, y_{2}\in Y\right).
 \end{equation}
 Then and only then there exist functions $\chi\colon X\to \mathbb{K}$ and
$\zeta\colon Y\to \mathbb{K}$ vanishing at zero such that
\[
 f(x, z)= \chi(x)+\zeta(z)
 \qquad
 \left(x\in X, z\in Y\right).
\]
Furthermore
\begin{enumerate}[(i)]
 \item either the following system of linear equations
 \[
  \begin{array}{rcl}
   \gamma_{1, 1}+\gamma_{1, 2}&=&0\\
   \gamma_{2, 1}+\gamma_{2, 2}&=&0
  \end{array}
 \]
is fulfilled or the function $\chi$ is identically zero.
\item either the following system of equations
 \[
  \begin{array}{rcl}
   \gamma_{1, 1}+\gamma_{2, 1}&=&0\\
   \zeta(\beta_{2}z)&=&\left(\gamma_{1, 2}+\gamma_{2, 2}\right)\zeta (z)
  \end{array}
  \qquad
  \left(z\in Y\right)
 \]
is fulfilled or the function $\zeta$ is identically zero.
\end{enumerate}
\end{prop}

\subsubsection{The case $\alpha_{1}, \alpha_{2}\neq 0$ and  $\beta_{1}, \beta_{2}=0$}

In such a situation  \eqref{Eq_non-trivial} reduces to
\[
 f(\alpha_{1}x_{1}+\alpha_{2}x_{2}, 0)= \gamma_{1, 1}f(x_{1},y_{1})+\gamma_{1, 2}f(x_{1},y_{2})+\gamma_{2, 1}f(x_{2},y_{1})+\gamma_{2, 2}f(x_{2},y_{2})
 \qquad
 \left(x_{1}, x_{2}\in X, y_{1}, y_{2}\in Y\right).
\]

As before, taking $f$ as the sum of single variable functions (Proposition \ref{Prop4}), substitute into the functional equation.

\begin{prop}\label{Prop_nontriv3}
 Let $X$ and $Y$ be linear spaces over the field $\mathbb{K}$, $\alpha_{1}, \alpha_{2}, \gamma_{i, j}\in \mathbb{K}$ be 
 given constants such that not all of them are zero
 and
 $f\colon X\times Y\to \mathbb{K}$ be a function such that
 \begin{equation}\label{Eq_nontriv_deg4}
  f(\alpha_{1}x_{1}+\alpha_{2}x_{2}, 0) = \gamma_{1, 1}f(x_{1},y_{1})+\gamma_{1, 2}f(x_{1},y_{2})+\gamma_{2, 1}f(x_{2},y_{1})+\gamma_{2, 2}f(x_{2},y_{2})
 \qquad
 \left(x_{1}, x_{2}\in X, y_{1}, y_{2}\in Y\right).
 \end{equation}
Then and only then there exists an additive function $a\colon X\to \mathbb{K}$ and a function $\zeta\colon Y\to \mathbb{K}$ vanishing at zero such that
\[
 f(x, y)= a(x)+\zeta(y)
 \qquad
 \left(x\in X, y\in Y\right).
\]
Furthermore the above additive function $a$ has to fulfill
\[
 a(\alpha_{1}x_{1}+\alpha_{2}x_{2})=(\gamma_{1, 1}+\gamma_{1, 2})a(x_{1})+(\gamma_{2, 1}+\gamma_{2, 2})a(x_{2})
\]
for arbitrary $x_{1}, x_{2}\in X$ and for the mapping $\zeta$ alternative (ii) of Proposition \ref{Prop_nontriv1} is fulfilled. 
\end{prop}

\subsubsection{The case $\alpha_{1}, \beta_{1}\neq 0$ and  $\alpha_{2}, \beta_{2}=0$}

In such a situation  \eqref{Eq_non-trivial} reduces to
\[
 f(\alpha_{1}x_{1}, \beta_{1}y_{1})= \gamma_{1, 1}f(x_{1},y_{1})+\gamma_{1, 2}f(x_{1},y_{2})+\gamma_{2, 1}f(x_{2},y_{1})+\gamma_{2, 2}f(x_{2},y_{2})
 \qquad
 \left(x_{1}, x_{2}\in X, y_{1}, y_{2}\in Y\right).
\]

To prove the following result, consider  $f$ as the sum of single variable functions (Proposition \ref{Prop5}) and substitute into the functional equation. 

\begin{prop}\label{Prop_nontriv4}
 Let $X$ and $Y$ be linear spaces over the field $\mathbb{K}$, $\alpha_{1}, \alpha_{2}, \gamma_{i, j}\in \mathbb{K}$ be given constants and
 $f\colon X\times Y\to \mathbb{K}$ be a function such that
 \begin{equation}\label{Eq_nontriv_deg5}
  f(\alpha_{1}x_{1}, \beta_{1}y_{1})= \gamma_{1, 1}f(x_{1},y_{1})+\gamma_{1, 2}f(x_{1},y_{2})+\gamma_{2, 1}f(x_{2},y_{1})+\gamma_{2, 2}f(x_{2}, y_{2})
 \qquad
 \left(x_{1}, x_{2}\in X, y_{1}, y_{2}\in Y\right).
 \end{equation}
Then and only then
\begin{enumerate}[(A)]
\item $\gamma_{1, 2}, \gamma_{2, 1}, \gamma_{2, 2}=0$ and 
$f\colon X\times Y\to \mathbb{K}$ is an arbitrary function fulfilling 
\[
  f(\alpha_{1}x, \beta_{1}y)= \gamma_{1, 1}f(x,y) 
  \qquad 
  \left(x\in X, y\in Y\right), 
\]

\item or there exist functions $\chi\colon X\to\mathbb{K}$ and $\zeta\colon Y\to \mathbb{K}$ vanishing at zero such that 
\[
 f(x, y)= \chi(x)+\zeta(y)
 \qquad
 \left(x\in X, z\in Y\right).
\]
Furthermore the mappings $\chi$ and $\zeta$ also fulfill
\[
 \begin{array}{rcl}
  \chi(\alpha_{1}x)&=&(\gamma_{1, 1}+\gamma_{1, 2})\chi(x)\\
  \zeta(\beta_{1}z)&=&(\gamma_{1, 1}+\gamma_{2, 1})\zeta(z)
 \end{array}
\qquad \left(x\in X, z\in Y\right)
\]
and 
\begin{enumerate}[(i)]
 \item either 
 \[
  \gamma_{2, 1}+\gamma_{2, 2}=0
 \]
or $\chi$ is identically zero; 
\item either 
\[
 \gamma_{1, 2}+\gamma_{2, 2}=0
\]
or $\zeta$ is identically zero. 
\end{enumerate}
\end{enumerate}
\end{prop}

\subsubsection{The case $\alpha_{1}, \alpha_{2}, \beta_{1}\neq 0$ and  $\beta_{2}=0$}

In such a situation  \eqref{Eq_non-trivial} reduces to
\[
 f(\alpha_{1}x_{1}+\alpha_{2}x_{2}, \beta_{1}y_{1})= \gamma_{1, 1}f(x_{1},y_{1})+\gamma_{1, 2}f(x_{1},y_{2})+\gamma_{2, 1}f(x_{2},y_{1})+\gamma_{2, 2}f(x_2, y_2)
 \qquad
 \left(x_{1}, x_{2}\in X, y_{1}, y_{2}\in Y\right).
\]

\begin{prop}\label{Prop_nontriv5}
 Let $X$ and $Y$ be linear spaces over the field $\mathbb{K}$, 
 $\alpha_{1}, \alpha_{2}, \beta_{1},  \gamma_{i, j}\in \mathbb{K}$, $i, j=1, 2$ be given constants and
 $f\colon X\times Y\to \mathbb{K}$ be a function such that
 \begin{multline}\label{Eq_nontriv_deg6}
  f(\alpha_{1}x_{1}+\alpha_{2}x_{2}, \beta_{1}y_{1})
  \\
  = \gamma_{1, 1}f(x_{1},y_{1})+\gamma_{1, 2}f(x_{1},y_{2})+\gamma_{2, 1}f(x_{2},y_{1})+\gamma_{2, 2}f(x_{2}, y_{2})
 \qquad
 \left(x_{1}, x_{2}\in X, y_{1}, y_{2}\in Y\right).
 \end{multline}
Then and only then 
there exists a mapping $A\colon X\times Y\to \mathbb{K}$ additive in its first variable, further there are functions 
$\chi\colon X\to \mathbb{K}$ and $\zeta\colon Y\to \mathbb{K}$ vanishing at zero such that $\chi$ is additive and 
\[
 f(x, y)= A(x, y)+\chi(x)+\zeta(y)
 \qquad 
 \left(x\in X, y\in Y\right). 
\]
Furthermore, we have that 
\[
 \zeta(\beta_{1}y)= (\gamma_{1, 1}+\gamma_{2, 1})\zeta(y) 
 \qquad 
 \left(y\in Y\right)
\]
and also
\[
 \left(\gamma_{1, 2}+\gamma_{2, 2}\right)\zeta(y)=0 
 \qquad 
 \left(y\in Y\right), 
\]
yielding that $\gamma_{1, 2}+\gamma_{2, 2}=0$ or $\zeta$ is identically zero. 
Additionally, the alternatives below also hold
\begin{enumerate}[A)]
 \item either $\gamma_{1, 2}$ and $\gamma_{2, 2}$ are zero, that is, equation \eqref{Eq_nontriv_deg6} has the form 
\[
  f(\alpha_{1}x_{1}+\alpha_{2}x_{2}, \beta_{1}y_{1})
    = \gamma_{1, 1}f(x_{1},y_{1})+\gamma_{2, 1}f(x_{2},y_{1})
  \qquad
 \left(x_{1}, x_{2}\in X, y_{1}\in Y\right)
\]
and the identities 
\[
 A(\alpha_{1}x, \beta_{1}y)+\chi(\alpha_{1}x)
 =
 \gamma_{1, 1}A(x, y)+\gamma_{1, 1}\chi(x)
 \qquad 
 \left(x\in X, y\in Y\right)
\]
and
\[
 A(\alpha_{2}x, \beta_{1}y)+\chi(\alpha_{2}x)
 =
 \gamma_{2, 1}A(x, y)+\gamma_{2, 1}\chi(x)
 \qquad 
 \left(x\in X, y\in Y\right)
\]
have to hold. 
\item or $\gamma_{1, 2}$ or $\gamma_{2, 2}$ do not vanish simultaneously and then the mapping $A$ has a rather special form, namely there exists an 
additive function $a\colon X\to \mathbb{K}$ such that 
\[
 A(x, y)= a(x)
 \qquad 
 \left(x\in X\right)
\]
and therefore 
\[
 f(x, y)= a(x)+\chi(x)+\zeta (y)
 \qquad
 \left(x\in X, y\in Y\right)
\]
where the identities 
\[
 a(\alpha_{1}x)+\chi(\alpha_{1}x)
 =
 \gamma_{1, 1}a(x)+\gamma_{1, 1}\chi(x)
 \qquad 
 \left(x\in X, y\in Y\right)
\]
and 
\[
 a(\alpha_{2}x)+\chi(\alpha_{2}x)
 =
 \gamma_{2, 1}a(x)+\gamma_{2, 1}\chi(x)
 \qquad 
 \left(x\in X, y\in Y\right)
\]
have to hold. 
\end{enumerate}
\end{prop}
\begin{proof}
 Using Proposition \ref{Prop6} we immediately get that there exists a mapping 
 $A\colon X\times Y\to \mathbb{K}$ and there are functions $\chi\colon X\to \mathbb{K}$ and 
 $\zeta\colon Y\to \mathbb{K}$ vanishing at zero such that $\chi$ is additive and  
 \[
  f(x, y)=A(x, y)+\chi(x)+\zeta(y)
  \qquad 
 \left(x\in X, y\in Y\right). 
 \]
Using that $A$ is additive in its first variable and equation \eqref{Eq_nontriv_deg6} we derive that 
\begin{multline}\label{Proof_Prop_nontriv5}
  A(\alpha_{1}x_{1}, \beta_{1}y_{1})+A(\alpha_{2}x_{2}, \beta_{1}y_{1})+\chi(\alpha_{1}x_{1}+\alpha_{2}x_{2})+\zeta(\beta_{1}y_{1})
  \\
    = \gamma_{1, 1}A(x_{1},y_{1})+\gamma_{1, 1}\chi(x_{1})+\gamma_{1, 1}\zeta(y_{1})
    +\gamma_{1, 2}A(x_{1}, y_{2})+\gamma_{1, 2}\chi(x_{1})+\gamma_{1, 2}\zeta(y_{2})
    \\
    +\gamma_{2, 1}A(x_{2},y_{1})+\gamma_{2, 1}\chi(x_{2})+\gamma_{2, 1}\zeta(y_{1})
    +\gamma_{2, 2}A(x_{2}, y_{2})+\gamma_{2, 2}\chi(x_{2})+\gamma_{2, 2}\zeta(y_{2})
    \\
 \left(x_{1}, x_{2}\in X, y_{1}, y_{2}\in Y\right)
\end{multline}
Observe that this equation with $x_{1}=x_{2}=y_{1}=0$ implies that 
\[
 \left(\gamma_{1, 2}+\gamma_{2, 2}\right)\zeta(y_{2})=0 
 \qquad 
 \left(y_{2}\in Y\right), 
\]
so $\gamma_{1, 2}+\gamma_{2, 2}=0$ or the function $\zeta$ is identically zero. 
Similarly, equation \eqref{Proof_Prop_nontriv5} yields with $x_{1}=x_{2}=y_{2}=0$ that 
\[
 \zeta(\beta_{1}y_{1})= (\gamma_{1, 1}+\gamma_{2, 1})\zeta(y_{1}) 
 \qquad 
 \left(y_{1}\in Y\right), 
\]
here while proving the last two identities we used that $\chi(0)=\zeta(0)=0$ (cf. the proof of Proposition \ref{Prop6})
and the fact that $A(0, y)=0$ for all $y\in Y$ due to that $A$ is additive in its first variable. 
Put $x_{2}=y_{1}=0$ into \eqref{Proof_Prop_nontriv5} to receive that 
\[
  A(\alpha_{1}x_{1}, 0)+\chi(\alpha_{1}x_{1})
  \\
    = \gamma_{1, 1}A(x_{1},0)+\gamma_{1, 1}\chi(x_{1})
    +\gamma_{1, 2}A(x_{1}, y_{2})+\gamma_{1, 2}\chi(x_{1})
    \qquad
 \left(x_{1}, x_{2}\in X, y_{1}, y_{2}\in Y\right)
\]
or in other words, 
\[
 -\gamma_{1, 2}A(x_{1}, y_{2}) 
  \\
    = \gamma_{1, 1}A(x_{1},0)+\gamma_{1, 1}\chi(x_{1})
    +\gamma_{1, 2}\chi(x_{1})
    -A(\alpha_{1}x_{1}, 0)-\chi(\alpha_{1}x_{1})
    \qquad 
 \left(x_{1}\in X, y_{2}\in Y\right)
\]
Similarly equation \eqref{Proof_Prop_nontriv5} with $x_{1}=y_{1}=0$ yields that 
\[
  A(\alpha_{2}x_{2}, 0)+\chi(\alpha_{2}x_{2})
  = \gamma_{2, 1}A(x_{2},0)+\gamma_{2, 1}\chi(x_{2})
    +\gamma_{2, 2}A(x_{2}, y_{2})+\gamma_{2, 2}\chi(x_{2})
  \qquad 
 \left(x_{2}\in X, y_{2}\in Y\right)
\]
or equivalently
\[
 -\gamma_{2, 2}A(x_{2}, y_{2})
 = \gamma_{2, 1}A(x_{2},0)+\gamma_{2, 1}\chi(x_{2})
    +\gamma_{2, 2}\chi(x_{2})
    -A(\alpha_{2}x_{2}, 0)-\chi(\alpha_{2}x_{2})
  \qquad 
 \left(x_{2}\in X, y_{2}\in Y\right). 
\]
From this latter two identities the following alternatives can be deduced
\begin{enumerate}[A)]
 \item either $\gamma_{1, 2}$ and $\gamma_{2, 2}$ are zero and 
 equation \eqref{Eq_nontriv_deg6} has the form 
\[
  f(\alpha_{1}x_{1}+\alpha_{2}x_{2}, \beta_{1}y_{1})
    = \gamma_{1, 1}f(x_{1},y_{1})+\gamma_{2, 1}f(x_{2},y_{1})
  \qquad
 \left(x_{1}, x_{2}\in X, y_{1}\in Y\right)
\]
and the identities 
\[
 A(\alpha_{1}x, \beta_{1}y)+\chi(\alpha_{1}x)
 =
 \gamma_{1, 1}A(x, y)+\gamma_{1, 1}\chi(x)
 \qquad 
 \left(x\in X, y\in Y\right)
\]
as well as 
\[
 A(\alpha_{2}x, \beta_{1}y)+\chi(\alpha_{2}x)
 =
 \gamma_{2, 1}A(x, y)+\gamma_{2, 1}\chi(x)
 \qquad 
 \left(x\in X, y\in Y\right)
\]
follow immediately from \eqref{Proof_Prop_nontriv5} with $x_{2}=y_{2}=0$ and $x_{1}=y_{2}=0$, respectively. 
 \item or the two-variable mapping $A$ can be represented as 
 \[
  A(x, y)= a(x) 
  \qquad 
  \left(x\in X, y\in Y\right). 
 \]
$A$ being additive in its first variable, this is possible if and only if $a\colon X\to \mathbb{K}$ is additive. 
This means that 
\[
 f(x, y)=a(x)+\chi(x)+\zeta(y)
 \qquad 
 \left(x\in X, y\in Y\right). 
\]
Using this representation and equation \eqref{Prop_nontriv5} first with $x_{2}=y_{2}=0$ and after that with  $x_{1}=y_{2}=0$ 
we get the identities 
\[
 a(\alpha_{1}x)+\chi(\alpha_{1}x)
 =
 \gamma_{1, 1}a(x)+\gamma_{1, 1}\chi(x)
 \qquad 
 \left(x\in X, y\in Y\right)
\]
and 
\[
 a(\alpha_{2}x)+\chi(\alpha_{2}x)
 =
 \gamma_{2, 1}a(x)+\gamma_{2, 1}\chi(x)
 \qquad 
 \left(x\in X, y\in Y\right). 
\]
\end{enumerate}

\end{proof}

\subsection{The non-degenerate case}

In view of the above results, now we can focus on the case
$\alpha_{1}, \alpha_{2}, \beta_{1}, \beta_{2}\neq 0$ and investigate
the existence of nontrivial solutions of functional equation
\begin{multline*}
 f\left(\alpha_{1}\,x_{1}+\alpha_{2}\,x_{2} , \beta_{1}\,y_{1}+\beta_{2}\,y_{2}\right)
 \\
  =\gamma_{1, 1}f(x_{1},y_{1})+\gamma_{1, 2}f(x_{1},y_{2})+\gamma_{2, 1}f(x_{2},y_{1})+\gamma_{2 ,2}f(x_{2},y_{2})
 \qquad
 \left(x_{1}, x_{2}\in X, y_{1}, y_{2}\in Y\right),
\end{multline*}
where for $i, j\in \left\{1, 2\right\}$, the constants $\gamma_{i,j}\in \mathbb{K}$ are given.

Here we will make use of the results of Section \ref{Sec_nondeg}, 
therefore (as in Section \ref{Sec_nondeg}) we always assume that the characteristic of the field $\mathbb{K}$ is
\emph{different} from $2$. As a direct application of Proposition \ref{Prop_nondeg} we derive the following. 

\begin{prop}\label{Prop_nontriv_nondeg}
 Let $X$ and $Y$ be linear spaces over the field
 $\mathbb{K}$, $\alpha_{i}, \beta_{j}, \gamma_{i, j}\in \mathbb{K}$, $i, j=1, 2$ be given constants and
 $f\colon X\times Y\to \mathbb{K}$ be a function such that
 \begin{multline}
  f\left(\alpha_{1}\,x_{1}+\alpha_{2}\,x_{2} , \beta_{1}\,y_{1}+\beta_{2}\,y_{2}\right)
 \\
  =\gamma_{1, 1}f(x_{1},y_{1})+\gamma_{1, 2}f(x_{1},y_{2})+\gamma_{2, 1}f(x_{2},y_{1})+\gamma_{2 ,2}f(x_{2},y_{2})
 \qquad
 \left(x_{1}, x_{2}\in X, y_{1}, y_{2}\in Y\right).
 \end{multline}
Then and only then, there exist a bi-additive function $A\colon
X\times Y\to \mathbb{K}$ and additive functions $\chi\colon X\to
\mathbb{K}$ and $\zeta\colon Y\to \mathbb{K}$ such that
\[
 f(x, z)= A(x, z)+\chi(x)+\zeta(z)
 \qquad
 \left(x\in X, z\in Y\right).
\]
Furthermore the following identities also have to be fulfilled
\[
 \begin{array}{rcl}
  \chi(\alpha_{1}x)&=&\left(\gamma_{1, 1}+\gamma_{1, 2}\right)\chi(x)\\
  \chi(\alpha_{2}x)&=& \left(\gamma_{2, 1}+\gamma_{2, 2}\right)\chi(x)\\
  \zeta(\beta_{1}z)&=&\left(\gamma_{1, 1}+\gamma_{2, 1}\right)\zeta(z)\\
  \zeta(\beta_{2}z)&=&\left(\gamma_{1, 2}+\gamma_{2, 2}\right)\zeta(z)\\
  A(\alpha_{i}x, \beta_{j}z)&=& \gamma_{i, j}A(x, y)
 \end{array}
\qquad \left(x\in X, z\in Y, i, j=1, 2\right).
\]
\end{prop}

\begin{rem}
While investigating whether equation \eqref{Eq_non-trivial} admits
or not a non-trivial solution we always got three type of
conditions. One of them is a purely algebraic condition, namely we
have to check if the parameters $\gamma_{i, j}$ fulfill a system of
homogeneous, linear equations.

The second type is about the existence of a non-trivial
semi-homogeneous additive function. More precisely, this condition
is always of the following form: let $X$ be a linear space over the
field $\mathbb{K}$ and let $a\colon X\to \mathbb{K}$ be an additive
function such that
\[
 a(\alpha x)=\beta a(x)
 \qquad
 \left(x\in X\right)
\]
with certain fixed scalars $\alpha, \beta \in \mathbb{K}$. For which
values of $\alpha$ and $\beta$ will the function $a$ be non-trivial
(that is, non-identically zero)? This question was firstly
investigated in Dar\'{o}czy \cite{Dar61} if $X=\mathbb{K}=\mathbb{R}$. 
These results were later
generalized and  extended in the papers \cite{Kis15, KisLac15,
KisVar14, KisVarVin15, Ste99, VarVin15, VinVar15}.
To the best of our knowledge, this problem has not been investigated 
in case of fields with nonzero characteristic. To this 
we provide a solution in Subsection \ref{Subsec_prime}. 

Our third condition is similar to the second one, namely it concerns
the non-triviality of a semi-ho\-mo\-ge\-neous bi-additive function.
The attached existence problem will also be discussed in the last section. 
\end{rem}

\section{A necessary and sufficient condition for the existence of non-zero, bi-additive semi-homogeneous mappings}
\label{semibiadd}

According to the characteristic property
\[
A(\alpha_{i}x, \beta_{j}y)= \gamma_{i, j}A(x, y) \qquad  \left(x\in X, y\in Y, i, j=1, 2\right)
\]
of the bi-additive term in the solution (see Proposition \ref{Prop_nontriv_nondeg}), it is natural to
investigate the problem of the existence of such  non-zero mapping
over the field $\mathbb{K}$.

\begin{dfn}
Let $X$ and $Y$ be linear spaces over the field $\mathbb{K}$. 
An additive function $a\colon X\to \mathbb{K}$ is called 
\emph{semi-homogeneous} if there exist elements $\alpha$, $\beta \in \mathbb{K}$  such that 
\[
 a(\alpha x)= \beta a(x) 
 \qquad 
 \left(x\in X\right). 
\]

Similarly, a bi-additive function $A\colon X\times Y\to \mathbb{K}$ is called
\emph{semi-homogeneous} if there exist elements $\alpha$, $\beta$
and $\gamma$ in the field $\mathbb{K}$  such that
\begin{equation}
\label{semihomokos} A(\alpha x, \beta y)=\gamma A(x,y) \qquad (x\in
X, y\in Y).
\end{equation}
\end{dfn} 

\subsection{The case of fields with characteristic zero}

In this subsection we restrict ourselves to the case of fields with zero characteristic. 
According to this, let $\mathbb{K}$ be a field of characteristic zero.
Then it is an extension of the field $\mathbb {Q}$ of the rationals
and we can consider the subfields $\mathbb{Q}(\alpha)$ and
$\mathbb{Q}(\beta)$ in $\mathbb{K}$.

We will also use the following. 

\begin{prop}
 Let $\mathbb{K}$ be a field of characteristic zero. 
 Then $\mathbb{K}$ is embeddable into $\mathbb{C}$ if and only if 
 the transcendence degree of the field extension $\mathbb{K}/\mathbb{Q}$ is less than $\mathfrak{c}$. 
\end{prop}

\begin{rem}
From the previous proposition we immediately get that if $\mathbb{K}$ is finitely generated over 
$\mathbb{Q}$, that is if $\mathbb{K}=\mathbb{Q}(\alpha_{1}, \ldots, \alpha_{n})$, then 
$\mathbb{K}$ is embeddable into $\mathbb{C}$. 
\end{rem}

\begin{thm}\label{Thm3} Let $\mathbb{K}$ be a field of characteristic zero and suppose
that $X$ and $Y$ are linear spaces over $\mathbb{K}$. There exists a
non-zero bi-additive mapping $A\colon X\times Y\to \mathbb{K}$
satisfying \eqref{semihomokos} if and only if $\gamma$ can be
written as the product of algebraically conjugated elements to
$\alpha$ and $\mathbb{\beta}$ over $\mathbb{Q}$, respectively.
\end{thm}

\begin{proof} Suppose that $\gamma$ can be written as the product of algebraically conjugated elements to $\alpha$ and $\mathbb{\beta}$ over $\mathbb{Q}$, respectively. This means that $\gamma=\delta_1(\alpha)\delta_2(\beta),$
where $\delta_1 \colon \mathbb{Q}(\alpha)\to \mathbb{K}$ and
$\delta_2 \colon \mathbb{Q}(\beta)\to \mathbb{K}$ are injective
homomorphisms. Taking $X$ and $Y$ as linear spaces over
$\mathbb{Q}(\alpha)$ and $\mathbb{Q}(\beta)$, respectively, we can
write that
\[
x=\sum_{i\in I}\frac{p_i(\alpha)}{q_i(\alpha)}x_i \quad \text{and}\quad  y=\sum_{j\in J}\frac{r_j(\beta)}{s_j(\beta)}y_j,
\]
where
\begin{enumerate}[(i)]
\item $I$ and $J$ are finite index sets such that $|I|=k$ and $|J|=l$,
\item $p_i$, $q_i\in \mathbb{Q}[x]$ for any $i\in I$,
\item $r_j$, $s_j\in \mathbb{Q}[x]$ for any $j\in J$,
\item $x_1, \ldots, x_k$ belong to a basis of $X$ as a linear space over $\mathbb{Q}(\alpha)$,
\item $y_1, \ldots, y_l$ belong to a basis of $Y$ as a linear space over $\mathbb{Q}(\beta)$.
\end{enumerate}
The mapping $A$ is defined by the formula of the semi-linear
extension
\[
A(x,y)=\sum_{i=1}^k
\sum_{j=1}^l\frac{p_i(\delta_1(\alpha))}{q_i(\delta_1(\alpha))}\frac{r_j(\delta_2(\beta))}{s_j(\delta_2(\beta))}A(x_i,
y_j)
\]
and the values $A(x_i, y_j)$ are not all zero.  

Conversely, suppose that there exists a
non-zero bi-additive mapping $A\colon X\times Y\to \mathbb{K}$
satisfying \eqref{semihomokos}. Let us fix elements $x\in X$ and
$y\in Y$ such that $A(x,y)\neq 0$. Taking the field
$\mathbb{L}=\mathbb{Q}(\alpha, \beta, A(x,y),\gamma)$ we can define
the bi-additive mapping $B\colon \mathbb{L}\times \mathbb{L}\to \mathbb{C}$ as
\[
B(u,v)=A(ux,vy) \qquad \left(u, v\in \mathbb{L}\right). 
\]
It can be easily seen that
\begin{equation}
\label{szemihomokos01} B(\alpha u, \beta v)=\gamma B(u,v)
\end{equation}
is fulfilled for arbitrary $u, v \in \mathbb{L}$.

Let $\mathbb{L}^{\ast}$ denote the multiplicative subgroup of
$\mathbb{L}$ and let $G=\mathbb{L}^*\times \mathbb{L}^*$ be the
group equipped with the pointwise multiplication. Then, for any
$(u_*,v_*)\in G$, the translate mapping, that is,
\[
(\tau_{(u^*,v^*)}B)(u,v)=B(uu^*,vv^*)
\]
also satisfies \eqref{szemihomokos01}. Let $V$ be the set of the
restrictions of bi-additive mappings of the form $B\colon\mathbb{L}\times \mathbb{L} \to \mathbb{C}$ satisfying
\eqref{szemihomokos01}. 
Then the set $V$ is closed with respect to the uniform convergence on finite sets and the field $\mathbb{L}$ is
countable. 
Therefore $V$ is a closed, translation invariant linear space, in other words, 
it is a variety over a field of finite transcendence degree.
From this we infer that spectral analysis holds in $V$, i.e., there
exists an exponential element in this variety, see Laczkovich--Székelyhidi \cite{LacSze05}. 
An exponential element in this variety is a bi-additive mapping
$M\colon \mathbb{L}\times \mathbb{L}\to \mathbb{C}$ satisfying
\eqref{szemihomokos01} such that
\[
M(uu_*,vv_*)=M(u,v)M(u_*,v_*).
\]
Using the notations $\delta_1(u)=M(u,1)$ and $\delta_2(v)=M(1,v)$,
it follows that $M(u,v)=\delta_1(u)\delta_2(v),$ where $\delta_1$
and $\delta_2$ are injective (field-) homomorphisms of $\mathbb{L}$.
Using property \eqref{szemihomokos01}
\[
\gamma M(u,v)=M(\alpha u, \beta v)=\delta_1(\alpha)\delta_2(\beta)M(u,v).
\]
Therefore $\gamma=\delta_1(\alpha)\delta_2(\beta)$ as  it was stated.
\end{proof}

\subsection{The case of finite fields of non-zero characteristic}\label{Subsec_prime}

If $\mathbb{K}$ is a field of characteristic different from zero then it is a \emph{field of prime characteristic}.  
First of all we collect those results from the theory of finite fields, that we intend to use subsequently. 
Here we rely on the monograph Lidl--Niederreiter \cite{LidNie86}. For any field $\mathbb{K}$, 
there is a minimal subfield, namely the \emph{prime field of $\mathbb{K}$}, 
which is the smallest subfield containing $1$. 
It is isomorphic either to $\mathbb{Q}$ (if the characteristic is zero), 
or to a finite field of prime order $\mathbb{Z}_{p}$ (in case $\mathrm{char}(\mathbb{K})=p$). Moreover, if $p$ is a prime and $n\in \mathbb{N}$ is arbitrary then 
up to an isomorphism there exists exactly one \emph{finite} field of order 
$q=p^{n}$. This field is nothing but the splitting field of the 
polynomial $x^{q}-x$ over $\mathbb{Z}_{p}$. This field is denoted by 
$\mathrm{GF}(q)$. 

Let now $a\colon X\to \mathbb{K}$ be a semi-homogeneous additive function, that is, assume that for the 
additive function $a$ we have 
\[
 a(\alpha x)= \beta a(x) 
 \qquad 
 \left(x\in X\right). 
\]

As we have seen in the proof of Theorem \ref{Thm3}, 
the problem of the existence of semi-homogeneous mappings defined on the linear space $X$ 
can be reduced to the problem of the existence of semi-homogeneous mappings defined on $\mathbb{K}$. 
It can be easily seen that the additivity automatically implies  
the homogeneity with respect to the multiplication by the elements of the prime field. 
By the argument of \cite{VarVin09}, it also follows that there exists an automorphism between 
the extensions of the prime field with $\alpha$ and $\beta$, respectively, 
such that it maps $\alpha$ into $\beta$. 
Conversely,  such an automorphism allows us to use the technique of the semi-linear extension to construct semi-homogeneous additive mappings. 
This criteria for the existence of semi-homogeneous additive mappings \emph{does not depend} on the characteristic of the fields 
but it is worth to investigate the problem of the subfields in $\mathbb{K}$ in some special cases as follows.

Let  $\varphi\colon \mathbb{K}\to \mathbb{K}$ be an automorphism of $\mathbb{K}$ with $\varphi(\beta)=\alpha$. 
Then 
\[
 (\varphi\circ a)(\alpha x)=\varphi(\beta)\cdot (\varphi\circ a)(x) = \alpha \cdot (\varphi\circ a)(x)
 \qquad 
 \left(x\in X\right). 
\]
This means that
\begin{enumerate}[(i)]
\item we have to guarantee the existence of an automorphism $\varphi\colon \mathrm{GF}(p^{n})\to \mathrm{GF}(p^{n})$ 
for which 
\[
 \varphi(\beta)= \alpha
\]
is satisfied;
\item we have to determine the homogeneity field (see Definition \ref{def_homfield}) of the additive mapping $\varphi\circ a\colon X\to \mathbb{K}$. 
\end{enumerate}

Suppose that $\mathbb{K}\simeq \mathrm{GF}(p^{n})$ for 
some prime $p$ and $n\in \mathbb{N}$.

To answer the above questions, for (i) we have to know the automorphism group of $\mathrm{GF}(p^{n})$, 
while for (ii)  we have to describe the sub-fields of $\mathrm{GF}(p^{n})$.

\begin{dfn}
Let $p$ be a prime and $n\in \mathbb{N}$. 
 By an \emph{automorphism $\varphi$ of $\mathrm{GF}(p^{n})$ over $\mathrm{GF}(p)$} we mean 
 an automorphism of $\mathrm{GF}(p^{n})$ that fixes the elements of $\mathrm{GF}(p)$. 
 More precisely, we require $\varphi$ to be a one-to-one mapping from $\mathrm{GF}(p^{n})$ 
 onto itself with 
 \[
 \begin{array}{rcl}
  \varphi(a+b)&=&\varphi(a)+\varphi(b) \\
 \varphi(ab)&=&\varphi(a)\varphi(b)
 \end{array}
 \qquad 
 \left(a, b\in \mathrm{GF}(p^{n})\right)
 \]
and 
\[
 \varphi(a)=a 
 \qquad 
 \left(a\in \mathrm{GF}(p)\right). 
\]
\end{dfn}

\begin{thm}
 Let $p$ be a prime and $n\in \mathbb{N}$. 
 The distinct automorphisms of $\mathrm{GF}(p^{n})$ over $GF(p)$ are exactly the mappings 
 $\varphi_{0}, \varphi_{1}, \ldots, \varphi_{n-1}$ defined by 
 \[
  \varphi_{j}(a)=a^{p^{j}} 
  \qquad 
  \left(a\in \mathrm{GF}(p^{n}), j=0, 1, \ldots, n-1\right). 
 \]
\end{thm}

\begin{rem}
 In other words, the above theorem says that the automorphism group of 
 $\mathrm{GF}(p^{n})$ over $\mathrm{GF}(p)$ is a cyclic group of order $n$ generated by $\varphi_{1}$. 
\end{rem}

Let $X$ be a linear space over the (not necessarily finite) field $\mathbb{K}$
and $a\colon X\to \mathbb{K}$ be an additive function. 
Then clearly, for any $k\in \mathbb{Z}$ we have 
\[
 a(kx)=ka(x)
 \qquad 
 \left(x\in X\right). 
\]
Nevertheless, it may happen that $a$ satisfies the same identity for all $x\in X$ and for some $\alpha\in \mathbb{K}\setminus \mathbb{Z}$, 
therefore we introduce the following.

\begin{dfn}\label{def_homfield}
Let $X$ be a linear space over the (not necessarily finite) field $\mathbb{K}$
and $a\colon X\to \mathbb{K}$ be an additive function and 
\[
 \mathbb{H}_{a}= \left\{\alpha\in \mathbb{K}\, \vert \, a(\alpha x)= \alpha a(x) \text{ for all } x\in X\right\}. 
\]
This set is called the \emph{homogeneity field} of the additive function $a$. Observe that this term is well-motivated, since we have the following. 
\end{dfn}

Although the following two statements are known in case $\mathbb{K}=\mathbb{R}$ (see Kuczma \cite{Kuc09}), for the sake of completeness  we present a short argument for them. 

\begin{prop}
 Let $X$ be a linear space over the  field $\mathbb{K}$
and $a\colon X\to \mathbb{K}$ be an additive function. 
Then $\mathbb{H}_{a}\subset \mathbb{K}$ is a field. 
\end{prop}
\begin{proof}
 Let $\alpha, \beta\in \mathbb{H}_{a}$, then 
 \[
  a((\alpha-\beta)x)= a(\alpha x)-a(\beta x)= \alpha a(x)-\beta a(x)= (\alpha-\beta)a(x) 
  \qquad 
  \left(x\in X\right), 
 \]
 yielding that $\alpha-\beta \in \mathbb{H}_{a}$. 
 Similarly, if $\beta \neq 0$, then 
 \[
  \alpha a(x)= a(\alpha x)= a\left(\beta \dfrac{\alpha}{\beta}x\right)= \beta a\left(\dfrac{\alpha}{\beta}x\right) 
  \qquad 
  \left(x\in X\right), 
 \]
 from which $\dfrac{\alpha}{\beta}\in \mathbb{H}_{a}$ follows. 
\end{proof}

In some sense, the converse is also true, namely we have the proposition below. The proof is 
based on the existence of Hamel bases of linear spaces. Therefore, in any case it is needed, 
the Axiom of Choice is supposed to hold. 

\begin{prop}
 Let $X$ be a linear space over the field $\mathbb{K}$, let further  $\mathbb{L}\subset \mathbb{K}$ be a subfield of $\mathbb{K}$. 
 Then there exists an additive function $a\colon X\to \mathbb{K}$ such that $\mathbb{H}_{a}= \mathbb{L}$. 
\end{prop}
\begin{proof}
 Let $B$ be the Hamel basis of the linear space $(X, \mathbb{L}, +, \cdot)$, which (according to Corollary 4.2.1. of Kuczma \cite{Kuc09}) does exist. 
 Fix $c\in \mathbb{K}\setminus\left\{0\right\}$ and define the function $f\colon B\to \mathbb{K}$ by 
 \[
  f(x)=c 
  \qquad 
  \left(x\in B\right). 
 \]
Making use of Theorem 4.3.1 of Kuczma \cite{Kuc09}, there exists a homomorphism $a$ from $(X, \mathbb{L}, +, \cdot)$ to 
$(\mathbb{K}, \mathbb{L}, +, \cdot)$ such that we additionally have that $a\vert_{B}=f$. 
Clearly, $a$ is an additive function and 
\[
 a(\alpha x)= \alpha a(x)
 \qquad 
 \left(x\in X, \alpha \in \mathbb{L}\right). 
\]
Thus $\mathbb{L}\subset \mathbb{H}_{a}$. 

For the converse statement, let $x\in X$ be arbitrary, then 
$x=\sum_{i=1}^{n}\lambda_{i} b_{i}$, where $\lambda_{i}\in \mathbb{L}$ and $b_{i}\in B$ for all $i=1, \ldots, n$. 
Furthermore, 
\[
 a(x)= a\left(\sum_{i=1}^{n}\lambda_{i} b_{i}\right)= \sum_{i=1}^{n}\lambda_{i} a(b_{i})
 =\sum_{i=1}^{n}\lambda_{i}f(b_{i}) 
 =c\cdot \sum_{i=1}^{n}\lambda_{i}\in c\cdot \mathbb{L}, 
\]
or equivalently, $a(X)\subset c\cdot \mathbb{L}$.

Let now $\alpha \in \mathbb{H}_{a}$ and $b_{0}\in B$
be arbitrary, then 
\[
 a(\alpha b_{0})= \alpha a(b_{0})= \alpha f(b_{0})= \alpha c. 
\]
On the other hand, since $\alpha b_{0}\in X$, inclusion $a(X)\subset c\cdot \mathbb{L}$ implies that 
there exists $\lambda \in \mathbb{L}$ such that $a(\alpha b_{0})=\lambda c$. 
Since $c$ was to be chosen nonzero, this means that $\alpha= \lambda\in \mathbb{L}$. 
Therefore $\mathbb{H}_{a}\subset \mathbb{L}$. 
\end{proof}

\begin{thm}
Let $p$ be a prime and $n\in \mathbb{N}$. Then for all $d\vert n$, the field $\mathrm{GF}(p^{n})$ 
admits \emph{exactly} one subfield isomorphic to $\mathrm{GL}(p^{d})$ and 
$\mathrm{GL}(p^{n})$ has no other type of sub-fields. Furthermore, this subfield is the set of zeros of 
the polynomial $x^{p^{d}}-x$ in $\mathrm{GF}(p^{n})$. 
\end{thm}

Finally, we provide necessary and sufficient conditions for the existence 
of non-zero, bi-additive semi-homogeneous mappings. 
 
The relations among the elements $\alpha, \beta$ and $\gamma$ such that the semi-homogeneity equation \eqref{semihomokos} 
is satisfied for some non-zero bi-additive mapping $A\colon X\times Y\to \mathbb{K}$  are more implicit as we will see in what follows.

\begin{lem}
\label{lemmakey} 
Let $X$ and $Y$ be linear spaces over the field $\mathbb{K}$ and let $\alpha,\beta, \gamma\in \mathbb{K}$ be given  non-zero elements. 
There exists a not identically zero bi-additive mapping $A\colon X\times Y\to \mathbb{K}$ satisfying the semi-homogeneity equation \eqref{semihomokos} if and only if there exists a not identically zero bi-additive mapping $B\colon \mathbb{K}\times \mathbb{K}\to \mathbb{K}$  satisfying equation 
\begin{equation}\label{Eq21}
B(\alpha u, \beta v)=\gamma B(u,v) \quad (\gamma\neq 0)
\end{equation}
\end{lem}

\begin{proof} Suppose that $A\colon X\times Y\to \mathbb{K}$ satisfies the semi-homogeneity equation \eqref{semihomokos} and $A(x,y)\neq 0$ 
for a certain element $(x, y)\in X\times Y$. 
The bi-additive mapping $B\colon \mathbb{K}\times \mathbb{K}\to \mathbb{K}$ defined by 
\[
B(u,v)=A(ux, vy) 
\qquad 
\left(u, v\in \mathbb{K}\right)
\]
obviously satisfies equation \eqref{Eq21}. 
Conversely, suppose that $B\colon \mathbb{K}\times \mathbb{K}\to \mathbb{K}$ satisfies equation \eqref{Eq21}.  
Let  $\left\{x_{\mu}\right\}_{\mu \in \Gamma_{X}}$ and $\left\{y_{\nu}\right\}_{\nu\in \Gamma_{Y}}$ be Hamel bases in $X$ and $Y$, respectively. 
Taking the projections
\[
\pi_X^1 \colon X\times Y\to \mathbb{K} \quad \text{and} \quad \pi_Y^1 \colon X\times Y\to \mathbb{K}
\]
onto the first coordinate of the elements with respect to the given bases it follows that the mapping
$A\colon X\times Y\to \mathbb{K}$ defined by 
\[
 A(x,y)=B(\pi_X^1(x), \pi_Y^1(y))
 \qquad 
 \left(x\in X, y\in Y\right)
 \]
fulfills \eqref{semihomokos}.
\end{proof}

\begin{rem} 
Note that there is no need any additional condition for the cardinality of the field $\mathbb{K}$ to prove Lemma \ref{lemmakey}.
\end{rem}

From now on the results are strongly based on the cardinality condition for $\mathbb{K}$ being finite. 
Let $\mathbb{K}=\mathrm{GF}(q)$, where $q=p^n$ for some prime number $p\in \mathbb{N}$ and consider a (finite) basis
$b_0, \ldots, b_{n-1}$ of $\mathbb{K}$ over its prime field $\mathbb{Z}_p$. 
It is clear that
\begin{itemize}
\item [(H)] bi-additivity implies $\mathbb{Z}_p$-homogeneity for any bi-additive mapping $B\colon \mathbb{K}\times \mathbb{K}\to \mathbb{K}$. 
\end{itemize}
Since the translation $\tau_i\colon \mathbb{K}\to \mathbb{K}$ with respect to the multiplication 
by the $i$\textsuperscript{th} element of the given basis ($i=0, \ldots, n-1$), that is, 
\[
\tau_i(x)=b_i\cdot x
\qquad 
\left(x\in \mathbb{K}\right)
\]
is a linear transformation,  we can consider its matrix representation $M^i$ given by
\[
\tau_i(b_k)=\sum_{j=0}^{n-1} m^{(i)}_{jk}b_j,
\]
where for any possible indices we have $m^{(i)}_{jk}\in \mathbb{Z}_p$. 
According to property (H), a simple calculation shows that equation \eqref{Eq21} is equivalent to 
\[
\gamma B(b_k,b_l)=\sum_{i, j=0}^{n-1}\alpha_i\beta_j \sum_{r, s=0}^{n-1}m^{(i)}_{rk}m^{(j)}_{sl} B(b_r, b_s),
\]
where $k, l=0, \ldots n-1$, $\displaystyle{\alpha=\sum_{i=0}^{n-1} \alpha_i b_i}$ and 
$\displaystyle{\beta=\sum_{j=0}^{n-1} \beta_j b_j}$ with $\alpha_i, \beta_j\in \mathbb{Z}_p$. 

Let $\mathscr{M}_n(\mathbb{K})$ be the linear space of matrices of order $n$ over the field $\mathbb{K}$ and consider the linear mapping
\begin{equation}
\label{characteristicpol}
P_{\alpha, \beta}\colon \mathscr{M}_n(\mathbb{K})\ni X \mapsto Y=P_{\alpha, \beta}(X),
\end{equation}
where
$$y_{kl}=\sum_{i, j=0}^{n-1}\alpha_i\beta_j \sum_{r, s=0}^{n-1}m^{(i)}_{rk}m^{(j)}_{sl}x_{rs}.$$
In a more compact form 
\[
P_{\alpha, \beta}(X)=\sum_{i, j=0}^{n-1}\alpha_i \beta_j \left(M^{(i)}\right)^T X M^{(j)}.
\]
Equation \eqref{Eq21} is obviously satisfied if and only if $0\neq \gamma\in \mathbb{K}$ is an eigenvalue  of $P_{\alpha, \beta}$. 
The corresponding (non-zero) eigenvector $B\in \mathscr{M}_n(\mathbb{K})$ 
can be chosen as the matrix of a bi-linear mapping satisfying \eqref{Eq21}. 
To sum up, we can formulate the following result as the answer for the problem of 
existence of non-identically zero bi-additive mapping satisfying \eqref{semihomokos}. 

\begin{thm}
Let $\mathbb{K}=\mathrm{GF}(q)$, where $q=p^n$ for some prime number $p$ and $n\in \mathbb{N}$. 
Consider the polynomial 
\[
P(u,v,w)=\det \left(P_{u, v}-w\cdot \mathrm{id}\right)
\qquad 
\left(u, v, w\in \mathbb{K}\right), 
\]
where $\mathrm{id}$ stands for the identity mapping of the linear space of matrices of order $n$ over the field $\mathbb{K}$. Then the following assertions are equivalent
\begin{enumerate}[(i)]
\item there is a not identically zero bi-additive mapping satisfying the semi-homogeneity equation \eqref{Eq21},
\item the characteristic polynomial of the linear operator $P_{\alpha, \beta}$ is reducible over the field $\mathbb{K}$ by one of its non-zero roots,
\item $P(\alpha, \beta, \gamma)=0.$
\end{enumerate}
\end{thm}

\begin{rem} If the elements $\alpha$ and $\beta$ are given, 
then the possible $\gamma$'s are among the roots of the characteristic polynomial $P_{\alpha, \beta}$. 
The characteristic polynomial is independent of the choice of the basis $b_0, \ldots, b_{n-1}$. 
Moreover it is a polynomial over the prime field but the root $\gamma$ belongs to $\mathbb{K}$ in general. 
Setting the variables $\alpha$ and $\beta$ free, the roots of the multivariate polynomial $P$ is independent of the choice of the basis $b_0, \ldots, b_{n-1}$. 
In other words the algebraic variety 
\[
P(x,y,z)=0
\]
in $\mathbb{K}^3$ contains all possible triplets for the solution of the semi-homogeneity equation \eqref{semihomokos}. 
\end{rem}

\subsection{An example: the field $\mathrm{GF}(4)$}

The operations are summarized in the following tables:

\begin{center}
\begin{tabular}{c|cccc}
$+$ & 0&1&$a$&$1+a$\\
\hline
0& 0&1&$a$&$1+a$\\
1& 1&0&$1+a$&$a$\\
$a$& $a$&$1+a$&0&1\\
$1+a$& $1+a$&$a$&1&0\\
\end{tabular}
\quad 
\text{and} 
\quad 
\begin{tabular}{c|cccc}
$\cdot$ & $0$&$1$&$a$&$1+a$\\
\hline
$0$& $0$&$0$&$0$&$0$\\
$1$& $0$&$1$&$a$&$1+a$\\
$a$& $0$&$a$&$1+a$&$1$\\
$1+a$& $0$&$1+a$&$1$&$a$\\
\end{tabular}
\end{center}

Since $4=2^2$ it follows that $\mathrm{GF}(4)$ is a two-dimensional linear space over its prime field $\mathbb{Z}_2$. 
The basis we are going to use in the following is $b_0=1$, $b_1=a$. 
An easy computation shows that the translations $\tau_0$ and $\tau_1$ are represented by the matrices
\[
M^0=
\begin{pmatrix}
1 & 0\\
0 & 1 
\end{pmatrix} 
\quad
\text{and} 
\quad 
M^1=
\begin{pmatrix}
0 & 1\\
1& 1
\end{pmatrix},
\]
respectively. Choosing elements 
\[
\alpha= \alpha_0+\alpha_1\cdot a, \quad \beta=\beta_0+\beta_1\cdot a,
\]
where $\alpha_0, \alpha_1, \beta_0, \beta_1\in \mathbb{Z}_2$, 
it follows that
\[P_{\alpha, \beta}(X)=\alpha_0\beta_0 X+\alpha_0\beta_1 X 
\begin{pmatrix}
0 & 1\\
1& 1
\end{pmatrix}
+\alpha_1 \beta_0 
\begin{pmatrix}
0 & 1\\
1& 1
\end{pmatrix} X+\alpha_1\beta_1 
\begin{pmatrix}
0 & 1\\
1& 1
\end{pmatrix}
X
\begin{pmatrix}
0 & 1\\
1& 1
\end{pmatrix},
\]
where $\displaystyle{X=
\begin{pmatrix}
x_{00} & x_{01}\\
x_{10}& x_{11}
\end{pmatrix}}.$ 
Taking 
\[P_{\alpha, \beta}(X)=
\begin{pmatrix}
y_{00} & y_{01}\\
y_{10}& y_{11}
\end{pmatrix},
\]
a  direct computation shows that
$$y_{00}=\alpha_0 \beta_0 x_{00}+\alpha_0 \beta_1 x_{01}+\alpha_1 \beta_0 x_{10}+\alpha_1\beta_1 x_{11},$$
$$y_{01}=\alpha_0\beta_0 x_{01}+\alpha_0 \beta_1 (x_{00}+x_{01})+ \alpha_1\beta_0 x_{11}+\alpha_1\beta_1(x_{10}+x_{11}),$$
$$y_{10}=\alpha_0\beta_0 x_{10}+\alpha_0\beta_1x_{11}+\alpha_1\beta_0(x_{00}+x_{10})+\alpha_1 \beta_1(x_{01}+x_{11}),$$
$$y_{11}=\alpha_0\beta_0x_{11}+\alpha_0\beta_1 (x_{10}+x_{11})+\alpha_1\beta_0(x_{01}+x_{11})+\alpha_1\beta_1(x_{00}+x_{01}+x_{10}+x_{11}).$$
Therefore $P_{\alpha, \beta}$ is represented by the matrix
\[\begin{pmatrix}
\alpha_0 \beta_0 & \alpha_0 \beta_1&\alpha_1 \beta_0&\alpha_1 \beta_1\\
\alpha_0 \beta_1& \alpha_0 \beta_0+\alpha_0 \beta_1&\alpha_1 \beta_1&\alpha_1 \beta_0+\alpha_1 \beta_1\\
\alpha_1 \beta_0&\alpha_1 \beta_1&\alpha_0 \beta_0+\alpha_1 \beta_0&\alpha_0 \beta_1+\alpha_1\beta_1\\
\alpha_1\beta_1& \alpha_1 \beta_0+\alpha_1\beta_1& \alpha_0 \beta_1+\alpha_1\beta_1&\alpha_0 \beta_0+\alpha_0 \beta_1+\alpha_1 \beta_0+\alpha_1\beta_1
\end{pmatrix}
\]
\noindent
with respect to the basis
\[
B_{00}=\begin{pmatrix}
1 & 0 \\
0 & 0 
\end{pmatrix}, \quad 
B_{01}=\begin{pmatrix}
0 & 1\\
0 & 0 
\end{pmatrix}, 
\quad 
B_{10}=\begin{pmatrix}
0 & 0 \\
1 & 0
\end{pmatrix}, 
\quad 
B_{11}=\begin{pmatrix}
0 & 0 \\
0 & 1
\end{pmatrix}.
\]
If $\alpha=1+a$, i.e. $\alpha_0=\alpha_1=1$ and $\beta=a$, i.e. $\beta_0=0$, $\beta_1=1$ then we have the matrix
\[
\begin{pmatrix}
0&1&0&1\\
1&1&1&1\\
0&1&0&0\\
1&1&0&0
\end{pmatrix}
\]
and the characteristic polynomial is
\[
P_{1+a, a}(t)=\left(t-1\right)^{2}\cdot \left(1+t+t^2\right).
\]
This means that the possible choices are $\gamma=1, a$ or $1+a$, that is, 
if $X$ and $Y$ are linear spaces over the field $\mathbb{K}=\mathrm{GF}(4)$, 
then there exist not identically zero bi-additive mappings of the form $A\colon X\times Y\to \mathbb{K}$ such that
\[A((1+a)x, ay)=A(x,y) \qquad (x\in X, y\in Y),\]
\[A((1+a)x, ay)=aA(x,y) \qquad (x\in X, y\in Y),\]
or
\[A((1+a)x, ay)=(1+a)A(x,y) \qquad (x\in X, y\in Y).\]

\begin{rem}
 $\mathrm{GF}(4)$ seems to be rich of semi-homogeneous biadditive functions in case of $\alpha=1+a$ and $\beta=a$. 
 In general, if $\alpha$ and $\beta$ are fixed, then the characteristic polynomial is of degree $n^2$, i.e. 
 we have at most $n^2$ different possible values for $\lambda$. 
 It is a polynomial dependence on $n$ but the number of the elements in $\mathrm{GF}(p^n)$ increases exponentially. 
 Therefore the probability of a randomly chosen element in $\mathbb{K}$ to be a possible value for $\lambda$ tends to zero.  
\end{rem}

In this last section of the paper we investigated 
the existence of non-zero, bi-additive semi-ho\-mo\-ge\-ne\-ous mappings. 
If the field $\mathbb{K}$ is of characteristic zero or $\mathbb{K}$ is a finite field, 
we could provide necessary and sufficient conditions. 
At the same time, there exist infinite fields of prime characteristic (for example, the field of all rational functions over $\mathbb{Z}/p\mathbb{Z}$).
Therefore, we end this paper with two open problems.

\begin{opp}
 Let $p$ be a prime and $\mathbb{K}$ be an infinite field of characteristic $p$. 
 Further, let $X$ be linear spaces over $\mathbb{K}$ and 
 $a\colon X\to \mathbb{K}$ be an additive function. 
 Find necessary and sufficient conditions for $a$ to be a nontrivial, semi-homogeneous additive mapping.  
\end{opp}

\begin{opp}
 Let $p$ be a prime and $\mathbb{K}$ be an infinite field of characteristic $p$. 
 Further, let $X$ and $Y$ be linear spaces over $\mathbb{K}$ and 
 $A\colon X\times Y\to \mathbb{K}$ be a bi-additive function. 
 Find necessary and sufficient conditions for $A$ to be a nontrivial, semi-homogeneous bi-additive mapping.  
\end{opp}

\begin{ackn}
E.~Gselmann was supported by the Hungarian Scientific Research Fund
(OTKA) Grant K 111651. and also by the EFOP-3.6.1-16-2016-00022 project. 
The project is co-financed by the European Union and the European Social Fund.

G.~Kiss was supported by the Hungarian Scientific Research Fund (OTKA) Grant K 124749.

Cs.~Vincze was supported by EFOP 3.6.2-16-2017-00015. 
The project is co-financed by the European Union and the European Social Fund.
\end{ackn}


\end{document}

%% file: moti.pdf_t
\begin{picture}(0,0)%
\includegraphics{moti.pdf}%
\end{picture}%
\setlength{\unitlength}{1367sp}%
\begingroup\makeatletter\ifx\SetFigFont\undefined%
\gdef\SetFigFont#1#2#3#4#5{%
  \reset@font\fontsize{#1}{#2pt}%
  \fontfamily{#3}\fontseries{#4}\fontshape{#5}%
  \selectfont}%
\fi\endgroup%
\begin{picture}(12309,8838)(31,-8260)
\put( 46,-1951){\makebox(0,0)[lb]{\smash{{\SetFigFont{8}{9.6}{\rmdefault}{\mddefault}{\updefault}{\color[rgb]{0,0,0}$y_{2}$}%
}}}}
\put( 46,-6406){\makebox(0,0)[lb]{\smash{{\SetFigFont{8}{9.6}{\rmdefault}{\mddefault}{\updefault}{\color[rgb]{0,0,0}$y_{1}$}%
}}}}
\put(2701,-8161){\makebox(0,0)[lb]{\smash{{\SetFigFont{8}{9.6}{\rmdefault}{\mddefault}{\updefault}{\color[rgb]{0,0,0}$x_{1}$}%
}}}}
\put(9946,-8116){\makebox(0,0)[lb]{\smash{{\SetFigFont{8}{9.6}{\rmdefault}{\mddefault}{\updefault}{\color[rgb]{0,0,0}$x_{2}$}%
}}}}
\put(6481,-4561){\makebox(0,0)[lb]{\smash{{\SetFigFont{8}{9.6}{\rmdefault}{\mddefault}{\updefault}{\color[rgb]{1,0,0}$\left(\frac{x_{1}+x_{2}}{2}, \frac{y_{1}+y_{2}}{2}\right)$}%
}}}}
\put(11971,-6856){\makebox(0,0)[lb]{\smash{{\SetFigFont{8}{9.6}{\rmdefault}{\mddefault}{\updefault}{\color[rgb]{0,0,0}$x$}%
}}}}
\put(1306,299){\makebox(0,0)[lb]{\smash{{\SetFigFont{8}{9.6}{\rmdefault}{\mddefault}{\updefault}{\color[rgb]{0,0,0}$y$}%
}}}}
\end{picture}%